\documentclass[10pt]{article}
\usepackage{titling}
\usepackage{blindtext}
\usepackage {amssymb,latexsym}
\usepackage {amsmath}
\usepackage{graphicx}
\usepackage{comment}
\usepackage{amsthm}
\usepackage{relsize}
\usepackage[margin=1in]{geometry}
\usepackage[latin1]{inputenc}
\usepackage{fancyhdr}
\usepackage{pgfplots}
\pgfplotsset{compat=1.14}
\usepackage{amscd}
\usepackage{hyperref}
\usepackage{graphicx}
\usepackage{relsize}
\usepackage{newlfont}
\usepackage{amssymb}
\IfFileExists{youngtab.sty}{\usepackage{youngtab}}{}
\IfFileExists{young.sty}{\usepackage{young}}{}
\usepackage{stmaryrd} 
\usepackage{mathtools}
\usepackage{tikz}
\usetikzlibrary{shapes, backgrounds} 
\usepackage{tikz-cd}
\usepackage{amsthm}
\usepackage{authblk} 
\usepackage{enumerate,enumitem}
\usepackage{mathrsfs}
\usepackage[dvipsnames]{xcolor}
\IfFileExists{soul.sty}{\usepackage{soul}}{} 
\usepackage[scr=rsfs,cal=boondox]{mathalfa} 

\newtheorem {lemma} {Lemma}
\newtheorem {proposition} [lemma] {Proposition}
\newtheorem {theorem} [lemma] {Theorem}
\newtheorem {corollary} [lemma] {Corollary}
\newtheorem {definition}[lemma] {Definition}
\theoremstyle{definition}
\newtheorem{remark}[lemma]{Remark}
\newtheorem{example}[lemma]{Example}

\excludecomment{invisible}
\newcommand{\g}{\mathfrak{g}}
\newcommand{\wt}{\vartriangleright}
\newcommand{\bt}{\blacktriangleright}
\renewcommand{\r}{\mathfrak{r}}
\newcommand{\s}{\mathfrak{s}}
\renewcommand{\l}{\mathfrak{l}}

\newcommand{\R}{\mathfrak{R}}
\renewcommand{\L}{\mathfrak{L}}
\newcommand{\Der}{\mathrm{Der}}
\newcommand{\Aut}{\mathrm{Aut}}
\newcommand{\ad}{\mathrm{ad}}
\newcommand{\ten}{\otimes}
\renewcommand{\wp}{\diamond}
\newcommand{\ot}{\otimes}
\newcommand{\ab}[2]{\langle #1,#2\rangle}
\newcommand{\abd}[2]{\langle #1,#2\rangle^{\bullet}}
\DeclareFontFamily{U}{mathb}{}
\DeclareFontShape{U}{mathb}{m}{n}{<-5.5> mathb5 <5.5-6.5> mathb6 
<6.5-7.5> mathb7 <7.5-8.5> mathb8 <8.5-9.5> mathb9 <9.5-11> mathb10 
<11-> mathb12}{}
\DeclareSymbolFont{mathb}{U}{mathb}{m}{n}
\DeclareFontSubstitution{U}{mathb}{m}{n}
\DeclareMathSymbol{\bp}{\mathbin}{mathb}{"0C}

\makeatletter
\@addtoreset{equation}{section}

\makeatother

\usepackage{xcolor,framed}

\begin{document}
\setlength{\parindent}{0pt} 
\title {\textbf{An infinitesimal deformation of the post-Lie and post-Hopf algebra correspondence}}

\author[a]{Andrea Rivezzi}
\author[b]{Andrea Sciandra}
\author[a]{Thomas Weber}

\affil[a]{{\sl
{Mathematical Institute of Charles University}}

{\sl Sokolovsk\'a 49/83, 186 75 Prague 8, Czech Republic}

{~}}

\affil[b]{{\sl
{D\'epartement de Math\'ematiques, Universit\'e Libre de Bruxelles}}
	
{\sl Boulevard du Triomphe, B-1050 Bruxelles, Belgium}  }
  
\date{}
\maketitle
\vspace{-1cm}

\begin{abstract}
\noindent
We describe infinitesimal deformations of post-Lie algebras and post-Hopf algebras and prove that the universal enveloping algebra and primitive elements adjunction is compatible with the infinitesimal structure. In characteristic zero, this restricts to an equivalence for connected cocommutative infinitesimal post-Hopf algebras, extending the Cartier--Milnor--Moore theorem. We parametrize all infinitesimal post-Lie structures on $\mathfrak{sl}(2)$
and exhibit a second-order obstruction. 
From a geometric point of view, infinitesimal post-Lie algebras emerge from flat connections with covariantly constant torsion and an additional compatible symmetric $(1,2)$-tensor. This leads to explicit examples of infinitesimal post-Lie structures based on Heisenberg Lie algebras.
On the Hopf side, we classify infinitesimal structures above the trivial action, prove rigidity for group algebras, and describe the structures on Sweedler's Hopf algebra. The subadjacent construction induces a Hopf algebra deformation with fixed coproduct. Finally, we prove that the quadratic operad of infinitesimal post-Lie algebras is Koszul, using a filtered distributive law between the operads of Lie algebras and bi-magmas.
\end{abstract}

\tableofcontents

\section*{Introduction}

Post-Lie algebras were introduced by B.~Vallette in \cite{Vallette} in an operadic context, as the Koszul dual of commutative trialgebras. Since then, they have attracted considerable attention due to their connections with several areas of mathematics. In differential geometry, they arise naturally from flat connections with covariantly constant torsion \cite{ML}, and play an important role in the geometric numerical integration of differential equations on manifolds \cite{CE-FM-K}. More recently, post-Lie structures have also appeared in the renormalization of stochastic partial differential equations \cite{BST}. From an algebraic perspective, their study is closely related to Rota--Baxter operators, universal enveloping algebras, and solutions of the Yang--Baxter equation \cite{BGST,BDV,LST}. These different appearances suggest that post-Lie algebras provide a useful framework for understanding the interaction between Lie algebra structures and additional algebraic operations.

Recall that a post-Lie algebra is a Lie algebra $(\g,[\cdot,\cdot])$ endowed with an additional bilinear operation $\wt\colon\g\times\g\to\g$, such that, for every $x\in\g$, the map $y\mapsto x\wt y$ is a derivation of $\g$, satisfying
\[
\begin{aligned}
[x,y]\wt z={}&x\wt(y\wt z)-y\wt(x\wt z)-(x\wt y)\wt z+(y\wt x)\wt z
\end{aligned}
\]
for all $x,y,z\in\g$. When the Lie bracket is trivial, these conditions reduce to the defining identity of a pre-Lie algebra. Moreover, every post-Lie algebra $(\g,[\cdot,\cdot],\wt)$ gives rise to another Lie algebra structure on the same underlying vector space, defined by $[\cdot,\cdot]_{\wt}:=[\cdot,\cdot]+\wt-\wt^{\mathrm{op}},$
where $\wt^{\mathrm{op}}$ denotes the precomposition of $\wt$ with the flip. The resulting Lie algebra $\g_{\wt}$ is called the \emph{subadjacent Lie algebra}. In the particular case of a pre-Lie algebra, this construction recovers the usual Lie bracket associated with the pre-Lie product.

The relationship between post-Lie algebras and Hopf algebras has been developed through the notion of a \emph{post-Hopf algebra}, introduced by Y.~Li, Y.~Sheng, and R.~Tang in \cite{LST}. Given a post-Lie algebra $(\g,[\cdot,\cdot],\wt)$, the operation $\wt$ extends to a coalgebra map $\wp\colon\mathrm{U}(\g)\ot\mathrm{U}(\g)\longrightarrow\mathrm{U}(\g),$
endowing its universal enveloping algebra with the structure of a cocommutative post-Hopf algebra \cite{ELM,LST}. More generally, a post-Hopf algebra consists of a Hopf algebra $H$, not necessarily cocommutative, equipped with a coalgebra map $\wp\colon H\ot H\to H$ satisfying suitable compatibility axioms (see Definition~\ref{def:PH}). In the cocommutative case, one can construct the \emph{subadjacent Hopf algebra} $H_{\wp}$, which has the same underlying coalgebra as $H$, but carries a new multiplication $\star_{\wp}$ and antipode $S_{\wp}$. This construction is compatible with the passage to subadjacent Lie algebras, in the sense that $\mathrm{U}(\g_{\wt})\cong\mathrm{U}(\g)_{\wp},$
see \cite{LST}.

An important feature of post-Hopf algebras is their connection with the quantum Yang--Baxter equation. Indeed, for a cocommutative post-Hopf algebra $(H,\wp)$, the linear map
\begin{equation}
\begin{aligned}
\sigma(a\ot b)
={}(a_1\wp b_1)\ot\bigl(S_{\wp}(a_2\wp b_2)\star_{\wp}a_3\star_{\wp}b_3\bigr)
\end{aligned}
\end{equation}
satisfies the braid equation, and consequently $\tau \circ \sigma$ is a solution of the quantum Yang--Baxter equation \cite{LST}. This is closely related to the equivalence between cocommutative post-Hopf algebras, Hopf braces \cite{AGV}, and matched pairs of actions on cocommutative Hopf algebras. These connections place post-Hopf algebras within the broader theory of quantum algebra and provide an additional motivation for investigating their algebraic and deformation-theoretic properties.

Conversely, the subspace $\mathrm{P}(H)$ of primitive elements of a post-Hopf algebra $(H,\wp)$ is naturally endowed with a post-Lie algebra structure, where the Lie bracket is given by the commutator and the post-Lie operation is the restriction of $\wp$. Thus, the universal enveloping algebra and primitive elements constructions determine an adjoint pair of functors
\begin{equation}\label{adj:LieHopf}
\begin{tikzcd}
\mathrm{PLie} \arrow[rr,bend left,"\mathrm{U}"] & \perp & \mathrm{PHopf} \arrow[ll,bend left,"\mathrm{P}"]
\end{tikzcd}
\end{equation}
between the categories of post-Lie and post-Hopf algebras. In characteristic zero, this adjunction restricts to an equivalence between post-Lie algebras and connected cocommutative post-Hopf algebras \cite{Catoire,ZD}, extending the classical Cartier--Milnor--Moore theorem \cite{MilnorMoore}. The importance of this correspondence is further illustrated by the categorical properties of cocommutative post-Hopf algebras. In particular, their category is semi-abelian \cite{GS,Semiabelian} and admits a torsion theory \cite{BournGran}, whose torsion subcategory consists of universal enveloping algebras of post-Lie algebras, while the torsion-free objects are group algebras associated with post-groups \cite{BGST}. The latter are closely related to skew braces \cite{GuVe}. In this way, the post-Lie/post-Hopf correspondence provides a unified perspective on several classes of algebraic objects.

Beyond these structural results, the post-Lie/post-Hopf correspondence suggests further questions concerning deformation and quantization. One particularly natural direction is motivated by the quantization of Lie bialgebras. Recall that a Lie bialgebra is a Lie algebra $(\g,[\cdot,\cdot])$ equipped with a Lie cobracket $\delta\colon\g\to\g\ot\g$ satisfying a compatibility condition with the Lie bracket. The quantization problem, introduced by Drinfeld \cite{Drinfeld90} and solved by Etingof and Kazhdan \cite{EK} (see also \cite{Sev}), consists in constructing an appropriate topological Hopf algebra whose semiclassical limit recovers the given Lie bialgebra structure. In view of the post-Lie/post-Hopf correspondence, it is natural to ask whether a similar quantization procedure can be developed in the presence of an additional post-Lie structure.

To motivate this question, suppose that $H_{\hbar}$ is a topological Hopf algebra over $\Bbbk[[\hbar]]$, equipped with a post-Hopf operation $\wp_{\hbar}$, such that $H_{\hbar}/\hbar H_{\hbar}\cong\mathrm{U}(\g)$
as Hopf algebras. Assuming the usual topological freeness conditions, the reduction of $\wp_{\hbar}$ modulo $\hbar$ defines a post-Hopf structure on $\mathrm{U}(\g)$, and hence a post-Lie structure on $\g$. On the other hand, the first-order antisymmetric part of the coproduct of $H_{\hbar}$ determines a Lie cobracket $\delta$ on $\g$. Consequently, the semiclassical limit carries both a post-Lie algebra structure and a Lie bialgebra structure. Since $\wp_{\hbar}$ is required to be a coalgebra morphism, these structures cannot be expected to be independent, but must satisfy appropriate compatibility conditions. This suggests the problem of determining which Lie bialgebras endowed with compatible post-Lie structures admit a quantization preserving the additional post structure. Related bialgebra-type constructions for post-Lie algebras have recently been investigated using Manin triples and generalized geometric structures \cite{LBG}.

A complete treatment of such a quantization problem would require understanding the interaction between the deformation of the Hopf algebra structure and that of the additional post-Hopf operation. As a preliminary step, it is therefore natural to consider the simpler problem of deforming the post structure while keeping the underlying Lie algebra or Hopf algebra fixed. This approach is analogous to the study of infinitesimal $\mathcal{R}$-matrices and infinitesimal braidings, where first-order deformations of additional algebraic structures lead to meaningful infinitesimal objects \cite{ABSW,ERSW}. More generally, first-order deformations play a fundamental role in deformation theory, as illustrated by the relationship between associative deformations and Poisson structures in deformation quantization \cite{BFFLS,Kontsevich}. Although infinitesimal deformations need not extend to full formal deformations, they identify the possible first-order directions and provide a starting point for investigating higher-order obstructions. In particular, the study of infinitesimal post structures is a natural problem in its own right, independently of the eventual construction of a quantization theory incorporating Lie cobrackets.

The infinitesimal deformation theory of post-Lie algebras has recently been studied from a cohomological and homotopical perspective by Lazarev, Sheng, and Tang \cite{LazST}. In particular, their framework allows simultaneous deformations of the Lie bracket and the post-Lie operation. In the present article, we consider a more specific deformation problem, in which the underlying Lie bracket is fixed and only the post-Lie operation is deformed. We investigate the corresponding problem for post-Hopf algebras, where the Hopf algebra structure remains undeformed. More precisely, we introduce the notions of \emph{infinitesimal post-Lie algebra} and \emph{infinitesimal post-Hopf algebra} by considering, respectively, a post-Lie algebra $(\g,[\cdot,\cdot],\wt)$ and a post-Hopf algebra $(H,\wp)$, together with bilinear maps
\[
\bt\colon\g\ot\g\to\g,
\qquad
\bp\colon H\ot H\to H,
\]
satisfying compatibility conditions such that
\begin{equation}
\widetilde{\wt}:=\wt+\hbar\bt,
\qquad
\widetilde{\wp}:=\wp+\hbar\bp
\end{equation}
define post-Lie and post-Hopf structures, respectively, over the ring $\Bbbk[\hbar]/(\hbar^2)$. The corresponding axioms are given in Definitions~\ref{def:infinitesimalPostLie} and~\ref{def:infinitesimalpostHopf}. Our objective is to understand these infinitesimal structures, their examples and rigidity properties, and, above all, their compatibility with the correspondence between post-Lie and post-Hopf algebras.

One of the main results of this article is that the universal enveloping algebra and primitive elements constructions extend to the infinitesimal setting. More precisely, we construct an adjunction
\begin{equation}
\begin{tikzcd}
\mathrm{IPLie} \arrow[rr,bend left,"\mathrm{U}"] & \perp & \mathrm{IPHopf} \arrow[ll,bend left,"\mathrm{P}"]
\end{tikzcd}
\end{equation}
between the categories of infinitesimal post-Lie and infinitesimal post-Hopf algebras (Theorem~\ref{thm:main}). Furthermore, over a field of characteristic zero, we prove that this adjunction becomes an equivalence of categories when restricted to connected cocommutative infinitesimal post-Hopf algebras. This provides an infinitesimal analogue of the Cartier--Milnor--Moore theorem and shows that the correspondence between post-Lie and post-Hopf algebras is preserved under the first-order deformations considered here.

In addition to this structural result, we investigate several examples and properties of infinitesimal post-Lie and post-Hopf algebras. After recalling the necessary preliminaries in Section~\ref{sec:Preliminaries}, we introduce infinitesimal post-Lie algebras in Section~\ref{sec:IPL}. We determine the infinitesimal post-Lie structures on $\mathfrak{sl}(2)$, building on the classification of post-Lie structures in terms of Rota--Baxter operators of weight one \cite{PLBG}. More precisely, we provide a parametrization of all compatible infinitesimal structures and describe the corresponding isomorphism relation (Theorem~\ref{thm:sl2-infinitesimal-classification}). We also show that some of these first-order deformations cannot be extended to formal deformations, by exhibiting an explicit second-order obstruction (Proposition~\ref{prop:sl2-obstructed-deformation}). For complete Lie algebras, we relate the space of infinitesimal post-Lie structures compatible with a fixed post-Lie product to a space of Chevalley--Eilenberg $1$-cocycles. We then discuss a geometric interpretation of infinitesimal post-Lie algebras arising from flat connections with covariantly constant torsion. In particular, we obtain a tensorial characterization of such infinitesimal structures, describe families of examples associated with Heisenberg groups, and establish rigidity results when the Lie algebras defined by the torsion are perfect.

In Section~\ref{sec:IPH}, we introduce infinitesimal post-Hopf algebras and investigate their main properties. We show that, for a cocommutative infinitesimal post-Hopf algebra, the associated subadjacent Hopf algebra is naturally endowed with a Hochschild $2$-cocycle (Theorem~\ref{thm:Hochschild}). Moreover, the resulting cocycle defines an infinitesimal Hopf algebra deformation with unchanged coproduct. We also study infinitesimal post-Hopf structures over the trivial post-Hopf action, describing them in terms of Hopf biderivations and the space of indecomposables. We prove that post-Hopf structures on group algebras are infinitesimally rigid and classify the infinitesimal post-Hopf structures on Sweedler's Hopf algebra. Finally, we establish the adjunction and the infinitesimal Cartier--Milnor--Moore theorem mentioned above.

Another main result of the article concerns the operadic description of infinitesimal post-Lie algebras. In Section~\ref{sec:operad}, we consider the quadratic operad $\mathcal{IPL}$ governing these structures and prove that it is Koszul (Theorem~\ref{thm:IPL}). This is achieved by constructing a filtered distributive law between the operad $\mathscr{L}$ of Lie algebras and the operad $\mathscr{M}_2$ of bi-magmas, yielding an isomorphism of $\mathbb{S}$-modules $\mathcal{IPL}\cong\mathscr{L}\circ\mathscr{M}_2.$
Besides providing a structural description of the operad, its Koszulity places infinitesimal post-Lie algebras within the framework of operadic homological algebra and opens the way to studying their associated deformation complexes. This also connects our results with the operadic origins of post-Lie algebras in \cite{Vallette}.

Finally, in Section~\ref{sec:outlook}, we discuss possible directions for future research. In particular, we return to the quantization problem motivating part of our study and the question of extending the post-Lie/post-Hopf correspondence to Lie bialgebras equipped with compatible post structures. We also consider possible generalizations to post-Lie-Rinehart algebras and post-Hopf algebroids, as well as further classification problems for infinitesimal post-Lie structures.

\section{Preliminaries}\label{sec:Preliminaries}
We recall some preliminary notions and results that will be useful throughout the paper. In the following $\Bbbk$ denotes a field, and we write $\otimes$ for the tensor product of $\Bbbk$-vector spaces. We denote the symmetric group on $n$ elements by $\mathbb{S}_n$, the set of $n,m$ shuffle permutations by $\mathbb{S}_{n,m}$, and the set of $n,m,k$ shuffle permutations by $\mathbb{S}_{n,m,k}$. We shall denote Hopf algebras by $H= (H,\mu,\eta,\Delta,\varepsilon,S)$. We sometimes write $\mu(a \ten b) = a \cdot b$ (or, more simply, $\mu(a \ten b)= ab$) and use the Sweedler notation $\Delta(a)=a_1 \ten a_2$. We denote the flip map of vector spaces by $\tau$.

\subsection{Post-Lie algebras}\label{sec:PLA} 
\subsubsection{Definition and main properties}
We first recall the definition of a post-Lie algebra introduced by B. Vallette in \cite{Vallette}.

\begin{definition}\label{def:PL}
A \textbf{post-Lie algebra} consists of a triple $(\g,[\cdot,\cdot], \wt)$, where $(\g,[\cdot,\cdot])$ is a Lie algebra, and $\wt: \g \times \g \to \g$ is a bilinear map satisfying
\begin{align}
x \wt [y,z] &= [x \wt y,z] + [y,x \wt z] \label{eq:postlie:1}\\
[x,y] \wt z &= x \wt (y \wt z) - y \wt (x \wt z) - (x \wt y) \wt z + (y \wt x) \wt z \label{eq:postlie:2}
\end{align}
for all $x,y,z \in \g$.  

A morphism between two post-Lie algebras $(\g,[\cdot,\cdot], \wt)$ and  $(\g',[\cdot,\cdot]', \wt')$ is a morphism of Lie algebras $f: (\g,[\cdot,\cdot]) \to (\g',[\cdot,\cdot]')$ satisfying $f(x \wt y) = f(x) \wt' f(y)$.

We shall denote the category of post-Lie algebras by $\mathrm{PLie}$.
\end{definition}

Given a triple $(\g,[\cdot,\cdot], \wt)$, where $(\g,[\cdot,\cdot])$ is a Lie algebra and $\wt: \g \times \g \to \g$ is a bilinear map, one can consider the linear map 
\begin{align}
D_x :\g \to \g, \quad y \mapsto x \wt y.
\end{align}
It is easy to see that Equation \eqref{eq:postlie:1} is equivalent to the fact that $D_x$ is a derivation of $\mathfrak{g}$, i.e., $D_x\in\Der(\g)$, for all $x\in\mathfrak{g}$. 
To any post-Lie algebra $(\g,[\cdot,\cdot], \wt)$ one assigns the Lie algebra $(\g,[\cdot,\cdot]_\wt)$, whose Lie bracket is 
\begin{align}
[x,y]_\wt := [x,y] + x \wt y - y \wt x,
\end{align}
the so-called \textbf{subadjacent Lie algebra}. In this way,
equation \eqref{eq:postlie:2} becomes equivalent to 
\begin{align}
\label{eq:interpretation-post-Lie}
[D_x, D_y] = D_{[x,y]_\wt}
\end{align}
for all $x,y\in\mathfrak{g}$. Therefore, \eqref{eq:postlie:1} and \eqref{eq:postlie:2} together are equivalent to the condition of $D: (\g,[\cdot,\cdot]_\wt) \to \Der(\g)$, $x \mapsto D_x$, being a morphism of Lie algebras.

\begin{example}
If $(\mathfrak{g},[\cdot,\cdot])$ is abelian, then $(\mathfrak{g},[\cdot,\cdot],\wt)$ is a post-Lie algebra if and only if $(\mathfrak{g},\wt)$ is a \textbf{pre-Lie algebra}, i.e., a vector space $\mathfrak{g}$ equipped with a bilinear map $\wt \colon \mathfrak{g}\times\mathfrak{g}\to\mathfrak{g}$ such that 
\begin{equation}\label{eq:preLie}
    \big(x\wt y)\wt z-(y\wt x)\wt z=x\wt(y\wt z)-y\wt(x\wt z)
\end{equation}
holds for all $x,y,z\in\mathfrak{g}$. Then, the Lie bracket of the subadjacent Lie algebra reads $[x,y]_\wt=x\wt y-y\wt x$ 
for all $x,y\in\mathfrak{g}$. Examples of pre-Lie algebras include associative algebras (with $\wt$ given by the product) and vector fields on a smooth manifold (where $\rhd$ is given by a flat, torsion-free connection).
\end{example}

\begin{remark}
\label{remark-simple}
Let $(\g,[\cdot,\cdot])$ be a finite-dimensional complete Lie algebra over a field of characteristic zero, i.e., a Lie algebra whose center is zero and whose derivations are all inner. Then it is easy to see that
$\g$ admits a post-Lie structure $\wt$ if and only if there exists a morphism of Lie algebras $\varphi: (\g,[\cdot,\cdot]_\wt) \to (\g, [\cdot,\cdot])$ such that $x \wt y = [\varphi(x),y]$. Note also that the condition of $\varphi: (\g,[\cdot,\cdot]_\wt) \to (\g, [\cdot,\cdot])$ being a Lie algebra morphism reads
\begin{equation}\label{eq:Rota-Baxter-operator}
[\varphi(x), \varphi(y)] = \varphi \big( [x,y] + [\varphi(x), y] + [x,\varphi(y)]\big).
\end{equation}
A morphism satisfying Equation \eqref{eq:Rota-Baxter-operator} is known in the literature as a \emph{Rota--Baxter operator of weight $1$}. 

Moreover, every such operator $\varphi$ induces a
representation of the subadjacent Lie algebra
$\g_\wt$ on $\g$, given by
\begin{equation}
\label{eq:rho-phi}
\rho_\varphi(x)(v)
=[\varphi(x),v]-\varphi([x,v]).
\end{equation}

A particularly relevant class of complete Lie algebras is the one of complex simple Lie algebras; in this case, there are three possibilities for the kernel of a Rota-Baxter operator $\varphi$:
\begin{itemize}
\item[(i)] $\ker(\varphi)=\g$. In this case one obtains $\wt = 0$, i.e., the trivial post-Lie structure on $\g$, which gives $[x,y]_\wt=[x,y]$.
\item[(ii)] $\ker(\varphi) = 0$. This is satisfied only when $\varphi = - \mathrm{id}$, i.e.\ when $x \wt y = - [x,y]$, which gives $[x,y]_\wt = -[x,y]$. 
\item[(iii)] $0 \subset \ker(\varphi) \subset\g$. This is the case where non-trivial post-Lie structures appear.  
Furthermore, it is possible to show that $(\g,[\cdot,\cdot]_\wt)$ is never simple (see e.g.\ \cite{BDV}). 
\end{itemize}
\end{remark}

In the next sections, we shall discuss the examples of post-Lie algebras we are mainly interested in, such as post-Lie structures on the complex simple Lie algebra $\s\l(2)$, and post-Lie structures arising from differential geometry.

\subsubsection{Classification of post-Lie structures on $\s\l(2)$}\label{subsec:postLiesl2}
\label{section-post-lie-sl2}

Denote by $\s\l(2)$ the Lie algebra of traceless
$2\times2$ complex matrices, with standard basis $e,f,h$ and relations
\begin{equation}\label{eq:sl2-relations}
[h,e]=2e,\qquad [h,f]=-2f,\qquad [e,f]=h.
\end{equation}
By Remark~\ref{remark-simple}, its post-Lie structures are precisely
$x\wt y=[\varphi(x),y]$, where $\varphi$ is a Rota--Baxter operator of weight $1$.
Two such structures are isomorphic precisely when their operators are
conjugate by an element of $\Aut(\s\l(2))$.
The complete classification of these products is given in
\cite[Theorem~1.2]{PLBG}. In the basis above, a complete set of
representatives consists of $\varphi=0$, $\varphi=-\mathrm{id}$, and
\begin{equation}\label{eq:sl2-normal-forms}
\begin{aligned}
\varphi_t(e)&=0,& \varphi_t(f)&=-f,& \varphi_t(h)&=th &&(t\in\mathbb C),\\
\varphi_*(e)&=0,& \varphi_*(f)&=-f-h,& \varphi_*(h)&=0,&&\\
\overline \varphi_*&=-\mathrm{id}-\varphi_* .&&&&&&
\end{aligned}
\end{equation}
These representatives are pairwise non-isomorphic. 
The image of the exceptional rank-one operator $\varphi_*$ is
spanned by the semisimple element $f+h$, whereas the image of $\varphi_0$ is spanned by the nilpotent element $f$.
\begin{remark}\label{rmk:sl2-old-parameters}
The family used in \cite[proof of Proposition~4.7]{BDV} is
\begin{equation}\label{eq:phi-sl2}
\varphi_{\alpha,\beta}(e)=0,\quad
\varphi_{\alpha,\beta}(f)=\frac{\alpha^2-\beta^2}{4}e-f-\frac\alpha2h,
\quad \varphi_{\alpha,\beta}(h)=-\beta e+\frac\beta{\alpha-\beta}h,
\end{equation}
where $\alpha\ne\beta$. The maps $\varphi_t$ and $\varphi_{\alpha,\beta}$ are related by
\begin{equation}\label{eq:sl2-conjugation}
\varphi_{\alpha,\beta}=G_s \circ \varphi_t \circ G_s^{-1},\qquad
t=\frac\beta{\alpha-\beta},\qquad s=\frac{\alpha-\beta}{2},
\qquad G_s=\exp(s\ad_e).
\end{equation}
Indeed, $G_s(e)=e$, $G_s(f)=f+sh-s^2e$ and $G_s(h)=h-2se$.
We also observe that \cite[Proposition~4.7]{BDV} classifies all possible
subadjacent Lie algebras, rather than all post-Lie products.
For $\varphi_t$ the subadjacent bracket is
$[h,e]_{\wt}=2(t+1)e$, $[h,f]_{\wt}=-2tf$ and
$[e,f]_{\wt}=0$; for $t\ne0$ its Lie algebra is
$\r_{3,-(t+1)/t}$, while for $t=0$ it is $\r_{3,0}$, 
where $\r_{3,\lambda}$ denotes the three-dimensional Lie algebra with basis $u,v,w$ and bracket
$[u,v]=v$, $[u,w]=\lambda w$ and $[v,w]=0$.
\end{remark}

\subsubsection{Post-Lie algebras arising from differential geometry}\label{sec:geompost}

We now recall the geometric origin of post-Lie algebras. In fact, they appear in differential geometry in the context of flat connections with covariantly constant torsion; see e.g.\ \cite[Sec. 2.1]{ML}.

Let $M$ be a smooth manifold together with a covariant derivative 
$$
\nabla\colon\Gamma^\infty(TM)\times\Gamma^\infty(TM)\to\Gamma^\infty(TM),
$$ 
where $\Gamma^\infty(TM)$ denotes the $\mathscr{C}^\infty(M)$-module of smooth sections of the tangent bundle $TM$. We view $\Gamma^\infty(TM)$ as an (infinite-dimensional) Lie algebra with respect to the Lie bracket $[X,Y]$ of vector fields $X,Y\in\Gamma^\infty(TM)$. The covariant derivative is function-linear in the first argument, i.e., $\nabla_{f\cdot X}Y=f\cdot\nabla_XY$, and satisfies the Leibniz rule $\nabla_X(f\cdot Y)=\mathscr{L}_X(f)\cdot Y + f\cdot\nabla_XY$, for all $f\in\mathscr{C}^\infty(M)$ and $X,Y\in\Gamma^\infty(TM)$, where $\mathcal{L}$ is the Lie derivative. Recall that $\nabla$ satisfies, for all $X,Y,Z\in\Gamma^\infty(TM)$, the \textit{first Bianchi identity}
\begin{equation}\label{eq:Bianchi}
    \mathrm{Tor}^\nabla(\mathrm{Tor}^\nabla(X,Y),Z)+\circlearrowright(X,Y,Z)
    =\mathrm{R}^\nabla(X,Y)Z-\nabla_X(\mathrm{Tor}^\nabla)(Y,Z)+\circlearrowright(X,Y,Z),
\end{equation}
where $\mathrm{Tor}^\nabla(X,Y):=\nabla_XY-\nabla_YX-[X,Y]$ denotes the \textit{torsion} of $\nabla$, $\mathrm{R}^\nabla(X,Y)Z:=\nabla_X\nabla_YZ-\nabla_Y\nabla_XZ-\nabla_{[X,Y]}Z$ is the \textit{curvature} of $\nabla$, and $\circlearrowright(X,Y,Z)$ stands for a sum of cyclic permutations in $X,Y,Z$.

Assuming that $\nabla$ is flat, i.e., that $\mathrm{R}^\nabla=0$, and that the torsion of $\nabla$ is covariantly constant, i.e., $\nabla_X(\mathrm{Tor}^\nabla)=0$ for all $X\in\Gamma^\infty(TM)$, it follows from \eqref{eq:Bianchi} that
\begin{equation}\label{eq:TorsionLie}
    -\mathrm{Tor}^\nabla\colon\Gamma^\infty(TM)\times\Gamma^\infty(TM)\to\Gamma^\infty(TM)
\end{equation} 
defines a Lie bracket on $\Gamma^\infty(TM)$ (the negative sign is not important for the Lie algebra structure, but turns out to be crucial in what follows). In this case,
\begin{equation}\label{eq:NablaAction}
    \rhd\colon\Gamma^\infty(TM)\times\Gamma^\infty(TM)\to\Gamma^\infty(TM),
    \qquad X\rhd Y:=\nabla_XY
\end{equation}
gives to $(\Gamma^\infty(TM),-\mathrm{Tor}^\nabla)$ the structure of a post-Lie algebra. In fact, \eqref{eq:postlie:1} is precisely the assumption that $\mathrm{Tor}^\nabla$ is covariantly constant, i.e.,
\begin{align*}
    0
    &=\nabla_X(\mathrm{Tor}^\nabla)(Y,Z)
    =\nabla_X(\mathrm{Tor}^\nabla(Y,Z))-\mathrm{Tor}^\nabla(\nabla_XY,Z)-\mathrm{Tor}^\nabla(Y,\nabla_XZ)\\
    &=X\rhd\mathrm{Tor}^\nabla(Y,Z)-\mathrm{Tor}^\nabla(X\rhd Y,Z)-\mathrm{Tor}^\nabla(Y,X\rhd Z),
\end{align*}
while flatness of $\nabla$ implies
\begin{align*}
    0
    &=\mathrm{R}^\nabla(X,Y)Z\\
    &=\nabla_X\nabla_YZ-\nabla_Y\nabla_XZ-\nabla_{[X,Y]}Z\\
    &=\nabla_X\nabla_YZ-\nabla_Y\nabla_XZ-\nabla_{\nabla_XY}Z+\nabla_{\nabla_YX}Z-\nabla_{-\mathrm{Tor}^\nabla(X,Y)}Z\\
    &=X\rhd(Y\rhd Z)-Y\rhd(X\rhd Z)
    -(X\rhd Y)\rhd Z+(Y\rhd X)\rhd Z
    -\nabla_{-\mathrm{Tor}^\nabla(X,Y)}Z,
\end{align*}
which is \eqref{eq:postlie:2}. In summary, $(\Gamma^\infty(TM),-\mathrm{Tor}^\nabla,\nabla)$ is a post-Lie algebra. The corresponding subadjacent Lie algebra
$$
-\mathrm{Tor}^\nabla(X,Y)+\nabla_XY-\nabla_YX=[X,Y]
$$
is the usual Lie algebra structure on $\Gamma^\infty(TM)$.

In case $\nabla$ is a flat equivariant connection on a Lie group $G$ with covariantly constant torsion, the above construction induces a post-Lie algebra structure on $\mathfrak{g}=T_eG$, the tangent space of $G$ at the neutral element, see e.g. \cite[Sec. 2.1]{LBG}. Namely, by the equivariance of $\nabla$, the maps \eqref{eq:TorsionLie} and \eqref{eq:NablaAction} induce maps $-\mathrm{Tor^\nabla\colon}\mathfrak{g}\times\mathfrak{g}\to\mathfrak{g}$, $\rhd\colon\mathfrak{g}\times\mathfrak{g}\to\mathfrak{g}$ and it is straightforward to show that $(\mathfrak{g},-\mathrm{Tor}^\nabla,\rhd)$ is a post-Lie algebra. The subadjacent Lie algebra is the usual Lie algebra structure on $\mathfrak{g}$ induced from the Lie group $G$.

\subsection{Post-Hopf algebras}

\subsubsection{Definition and main properties}

Recall that, for any Lie algebra $\mathfrak{g}$, the universal enveloping algebra $\mathrm{U}(\mathfrak{g})$ is naturally a Hopf algebra with coproduct, counit and antipode determined on primitive elements $x\in\mathfrak{g}$ by $\Delta(x)=x\otimes 1+1\otimes x$, $\varepsilon(x)=0$ and $S(x)=-x$. It is therefore natural to ask if there is an algebraic structure induced by a post-Lie algebra. The answer was given by Y. Li, Y. Sheng, and R. Tang \cite{LST} who defined the so-called post-Hopf algebras:

\begin{definition}\label{def:PH}
A \textbf{post-Hopf algebra} is a pair $(H,\wp)$, where $H$ is a Hopf algebra and $\wp:H\ot H\to H$ is a coalgebra morphism satisfying the following equalities for all $a,b,c \in H$:
\begin{align}
a\wp(b\cdot c)&=(a_{1}\wp b)\cdot(a_{2}\wp c)\label{wpproduct},\\
a\wp(b\wp c)&=(a_{1}\cdot(a_{2}\wp b))\wp c\label{wpdefassoci},
\end{align}
and such that the morphism $\alpha_{\wp}:H\to\mathrm{End}(H)$, $a\mapsto a\wp(-)$ is convolution invertible in $\mathrm{Hom}(H,\mathrm{End}(H))$, i.e.\ there exists a (necessarily) unique $\beta_{\wp}:H\to\mathrm{End}(H)$ such that 
\[
(\alpha_{\wp}a_{1})\circ(\beta_{\wp}a_{2})=(\beta_{\wp}a_{1})\circ(\alpha_{\wp}a_{2})=\varepsilon(a)\mathrm{id}_{H}.
\]
A morphism between two post-Hopf algebras $(H,\wp)$ and $(H',\wp')$ is a morphism of Hopf algebras $f:H\to H'$ such that $f(a\wp b)=f(a)\wp'f(b)$, for all $a,b\in H$. 

We shall denote the category of post-Hopf algebras by $\mathrm{PHopf}$.
\end{definition}

From the axiom of a post-Hopf algebra $(H,\wp)$ one immediately deduces that
$$
1\wp a=a,\qquad
a\wp 1=\varepsilon(a)1,\qquad
S(a\wp b)=a\wp S(b)
$$
hold for all $a,b\in H$, see also \cite[Lemma 2.3]{LST}.

\begin{remark}\label{rmk:subadj}
Given a cocommutative post-Hopf algebra $(H,\wp)$, one obtains the so called \textbf{subadjacent Hopf algebra} $H_{\wp}:=(H,\star_{\wp},\eta,\Delta,\varepsilon,S_{\wp})$ where:
\[
a\star_{\wp} b:=a_{1}\cdot(a_{2}\wp b),\qquad S_{\wp}(a):=\beta_{\wp}a_{1}(S(a_{2})),
\]
see \cite[Theorem 2.5]{LST}. Furthermore, $(H,\mu,\eta,\Delta,\varepsilon,S)$ is an object in $\mathrm{Hopf}(_{H_{\wp}}\mathcal{M})$, i.e., it is a Hopf monoid in the braided monoidal category of left $H_{\wp}$-modules, where the chosen braiding on $_{H_{\wp}}\mathcal{M}$ is the flip map $\tau$ (we recall that $H_{\wp}$ is cocommutative). The cocommutativity assumption is also crucial in the realm of the \textit{quantum Yang--Baxter equation} (QYBE). Indeed, as proven in \cite[Theorem 4.8]{LST}, given a cocommutative post-Hopf algebra $(H,\wp)$, the following is a solution to the braid equation (equivalently, $\tau\sigma$ satisfies the QYBE) on the vector space $H$:
\begin{equation}
\sigma(a\ot b):=(a_{1}\wp b_{1})\ot(S_{\wp}(a_{2}\wp b_{2})\star_{\wp}a_{3}\star_{\wp}b_{3}),\qquad \text{for all}\ a,b\in H.
\end{equation}
This happens since cocommutative post-Hopf algebras are equivalent to cocommutative Hopf braces introduced in \cite{AGV}; see \cite[Theorem 2.13]{LSTarXiv} (which is the arXiv version of \cite{LST}). Moreover, cocommutative Hopf braces are also equivalent to matched pairs of actions on cocommutative Hopf algebras \cite[Theorem 3.3]{AGV}, hence producing solutions of the QYBE \cite[Corollary 2.4]{AGV}. Removing the cocommutativity assumption, matched pairs of actions on Hopf algebras are equivalent to Yetter--Drinfeld braces \cite{FeSc} and to Yetter--Drinfeld post-Hopf algebras \cite{Sciandra}, and all these structures produce solutions of the QYBE in the non-cocommutative setting. We point out that Yetter--Drinfeld post-Hopf algebras $(H,\wp)$ are not standard Hopf algebras, but are objects in $\mathrm{Hopf}(^{H_{\wp}}_{H_{\wp}}\mathcal{YD})$, i.e., they are Hopf monoids in the braided monoidal category of (left-left) Yetter--Drinfeld modules over the Hopf algebra $H_{\wp}$ defined as above.
\end{remark}

\begin{example}\label{ex:SweedlerpostHopf}
We recall the following two examples of post-Hopf algebras that will be useful in the following.
\begin{itemize}
    \item[1)] Any Hopf algebra $H$ is a post-Hopf algebra $(H,\wp)$, where $\wp$ is given by the (left) trivial action $\varepsilon\ot\mathrm{id}$. 
    \item[2)] Let $\Bbbk$ be a field of characteristic different from 2. Consider Sweedler's Hopf algebra $H_4=\mathrm{span}_\Bbbk\{1,g,x,gx\}$ with $g^2=1$, $x^2=0$, $xg=-gx$, $\Delta(g)=g\otimes g$ and $\Delta(x)=x\otimes 1+g\otimes x$. The post-Hopf algebra structures for $H_{4}$ were classified in \cite[Example 2.12]{LST}. We recall that, up to isomorphism, there are only two post-Hopf algebra structures given by
    \[
\begin{tabular}{c|cccc}
    $\wp_{0}$& $1$ & $g$ & $x$ & $gx$\\ \hline
    $1$ & $1$ & $g$ & $x$ & $gx$ \\ 
    $g$ & $1$ & $g$ & $-x$ & $-gx$ \\ 
    $x$ & $0$ & $0$ & $0$ & $0$ \\ 
    $gx$ & $0$ & $0$ & $0$ & $0$ \\ 
\end{tabular}\qquad
\begin{tabular}{c|cccc}
    $\wp_{1}$& $1$ & $g$ & $x$ & $gx$\\ \hline
    $1$ & $1$ & $g$ & $x$ & $gx$ \\ 
    $g$ & $1$ & $g$ & $-x$ & $-gx$ \\ 
    $x$ & $0$ & $0$ & $x$ & $gx$ \\ 
    $gx$ & $0$ & $0$ & $x$ & $gx$ \\ 
\end{tabular}
\]
apart from the trivial post-Hopf algebra structure $(H_{4},\varepsilon\ot\mathrm{id})$.
\end{itemize}
\end{example}
One can easily verify that the category $\mathrm{PHopf}$ inherits the symmetric monoidal structure from the category $\mathrm{Vec}_{\Bbbk}$ of $\Bbbk$-vector spaces.
\begin{proposition}\label{prop:PHopfbraided}
The category $\mathrm{PHopf}$ is symmetric monoidal.
\end{proposition}
\begin{proof}
Let $(H,\wp)$ and $(H',\wp')$ be two post-Hopf algebras. It is well-known that $H\ot H'$ is a Hopf algebra and that $\wp_{\ot}:=(\wp\ot\wp')(\mathrm{id}_{H}\ot\tau_{H',H}\ot\mathrm{id}_{H'})
    $ is a morphism of coalgebras (since $\wp$, $\wp'$ and $\tau_{H',H}$ are morphisms of coalgebras). Since \eqref{wpproduct} holds for $\wp$ and $\wp'$, it is easily verified that $\wp_{\ot}$ satisfies \eqref{wpproduct}.
Similarly, since \eqref{wpdefassoci} holds for $\wp$ and $\wp'$, it is easily verified that $\wp_{\ot}$ satisfies \eqref{wpdefassoci}.
We know that $\alpha_{\wp}:H\to\mathrm{End}(H)$, $a\mapsto a\wp(-)$ and $\alpha_{\wp'}:H'\to\mathrm{End}(H')$, $a'\mapsto a'\wp'(-)$ have convolution inverses given by $\beta_{\wp}:H\to\mathrm{End}(H)$ and $\beta_{\wp'}:H'\to\mathrm{End}(H')$, respectively.
It is straightforward to verify that $\beta_{\wp_{\ot}}:H\ot H'\to\mathrm{End}(H\ot H')$, $a\ot a'\mapsto\beta_{\wp}a\ot\beta_{\wp'}a'$ is the convolution inverse of $\alpha_{\wp_{\ot}}:H\ot H'\to\mathrm{End}(H\ot H')$, $a\ot a'\mapsto (a\ot a')\wp_{\ot}(-)$.
Therefore, $(H\ot H',\wp_{\ot})$ is a post-Hopf algebra. Moreover, clearly $(\Bbbk,\mu_{\Bbbk})$ is a post-Hopf algebra and $(\mathrm{PHopf},\otimes,\Bbbk)$ is a monoidal category with the same constraints of $(\mathrm{Vec}_{\Bbbk},\ot,\Bbbk)$.
Finally, denoting by $\wp'_{\ot}$ the morphism of the post-Hopf algebra $H'\ot H$, we have
\[
\tau((a\ot a')\wp_{\ot}(b\ot b'))=
(a'\wp' b')\ot(a\wp b)=(a'\ot a)\wp'_{\ot}(b'\ot b)=\tau(a\ot a')\wp'_{\ot}\tau(b\ot b'),
\]
hence $\tau$ is a morphism in $\mathrm{PHopf}$. 
\end{proof}

\subsubsection{Compatibility between post structures and universal enveloping algebras}
In what follows, we shall denote the subspace of the primitive elements of a Hopf algebra $H$ by $P(H)$. \medskip

It is well known, see e.g.\ \cite[Theorem 2.7]{LST}, that, given a post-Hopf algebra $(H,\wp)$, its subspace $P(H)$ of primitive elements is a post-Lie algebra, where the post structure is given by the restriction $\wt:=\wp|_{P(H)\ot P(H)}$. We also recall the following result (which goes back to \cite{ELM}).

\begin{theorem}[{\cite[Theorem 2.8]{LST}}]\label{thm:postHopfonuniversalenveloping}
Let $(\mathfrak{g},[\cdot,\cdot],\wt)$ be a post-Lie algebra. Then $(\mathrm{U}(\mathfrak{g}),\wp)$ is a post-Hopf algebra, where $\wp$ is the extension of $\wt$ determined by
\[
1\wp u=u,\qquad (x^{1}\cdots x^{r})\wp u=x^{1}\wp((x^{2}\cdots x^{r})\wp u)-(x^{1}\wp(x^{2}\cdots x^{r}))\wp u
\]
for all $x^{1},\ldots, x^{r}\in\g$, with $r\geq1$ and $u\in \mathrm{U}(\mathfrak{g})$. Moreover, the subadjacent Hopf algebra $\mathrm{U}(\mathfrak{g})_{\wp}$ is isomorphic to the universal enveloping algebra $\mathrm{U}(\mathfrak{g}_{\wt})$, where $\g_\wt$ denotes the subadjacent Lie algebra $(\mathfrak{g},[\cdot,\cdot]_{\wt})$.
\end{theorem}

In fact, there is an adjunction
\begin{equation}\label{ad:PLiePHopf}
	\begin{tikzcd}
		\mathrm{PLie} \arrow[rr, bend left, "\mathrm{U}"] & \perp & \mathrm{PHopf} \arrow[ll, bend left, "P"]
	\end{tikzcd}
	\end{equation}
extending the well-known one between the categories $\mathrm{Lie}$ and $\mathrm{Hopf}$.
More explicitly, for any post-Hopf algebra $(H, \wp)$ and post-Lie algebra $(\g,[\cdot,\cdot],\wt)$, the following pair of (mutually inverse) morphisms realizes such an adjunction:
\begin{align}
&\Phi:\mathrm{Hom}_{\mathrm{PHopf}}\big(\mathrm{U}(\mathfrak{g}),H\big) \to \mathrm{Hom}_{\mathrm{PLie}}\big(\mathfrak{g},P(H)\big), \quad 
f  \mapsto f|_{P(\mathrm{U}(\mathfrak{g}))}\circ\eta_{\mathfrak{g}} \label{def:Phi}\\
&\Psi:\mathrm{Hom}_{\mathrm{PLie}}\big(\mathfrak{g},P(H)\big) \to \mathrm{Hom}_{\mathrm{PHopf}}\big(\mathrm{U}(\mathfrak{g}),H\big), \quad  f \mapsto \Psi(f)\label{def:Psi},
\end{align}
where $\eta_{\mathfrak{g}}:\mathfrak{g}\to P(\mathrm{U}(\mathfrak{g}))$ is the canonical inclusion and $\Psi(f)$ is the unique morphism in $\mathrm{PHopf}$ (in fact, the unique morphism of algebras) that extends $\iota\circ f$, where $\iota:P(H)\hookrightarrow H$ is the canonical inclusion. 
We also recall that $\eta_{\mathfrak{g}}$ is the $\mathfrak{g}$-component of the unit of this adjunction while, given $H$ in $\mathrm{PHopf}$, the $H$-component $\epsilon_{H}:\mathrm{U}(P(H))\to H$ of the counit of this adjunction is the unique morphism of algebras which extends $P(H)\hookrightarrow H$. \medskip

Theorem~\ref{thm:postHopfonuniversalenveloping} produces several examples of post-Hopf algebras, such as the universal enveloping algebra of the free post-Lie algebra; see e.g.\ \cite{Foissy,Li}, and the post-Hopf structures on $\mathrm U(\mathfrak{sl}(2))$ induced by all post-Lie structures presented in Section~\ref{subsec:postLiesl2}.

\begin{remark}\label{rmk:MilnorMoorepostHopf}
We also observe that the \emph{Cartier--Milnor--Moore Theorem} \cite{MilnorMoore} was lifted at the level of cocommutative post-Hopf algebras; see \cite[Corollary 28]{Catoire} or also \cite[Theorem 4.21 and Corollaries 4.22 and 4.23]{ZD}. Given a cocommutative connected post-Hopf algebra $(H,\wp)$ over a field of characteristic 0, the $H$-component $\epsilon_H$ of the counit of the adjunction \eqref{ad:PLiePHopf} becomes an isomorphism in $\mathrm{PHopf}$. Moreover, under the same assumption on the base field, one has 
$\eta_{\g}=\mathrm{id}$, hence the adjunction \eqref{ad:PLiePHopf} becomes an equivalence of categories between cocommutative connected post-Hopf algebras and post-Lie algebras.
\end{remark}

\section{Infinitesimal post-Lie algebras and infinitesimal post-Hopf algebras}

\subsection{Infinitesimal post-Lie algebras}\label{sec:IPL}

We now introduce infinitesimal post-Lie algebras. As laid out in the introduction, the guiding idea follows the same lines of deformation quantization \cite{BFFLS} and the definition of infinitesimal $\mathcal{R}$-matrices for quasitriangular Hopf algebras \cite{ABSW}, where only \emph{one} algebraic structure, in the present case the post-Lie structure $\wt$, is deformed. We will see that natural examples arise, and we will obtain a corresponding notion of infinitesimal post-Hopf algebra for its universal enveloping algebra in the follow-up section.

\subsubsection{Definition and main properties}
\begin{definition}\label{def:infinitesimalPostLie}
An \textbf{infinitesimal post-Lie algebra} consists of a quadruple $(\g,[\cdot,\cdot], \wt,\bt)$, where $(\g,[\cdot,\cdot],\wt)$ is a post-Lie algebra, and $\bt: \g \times \g \to \g$ is a bilinear map satisfying
\begin{align}
x \bt [y,z] &= [x \bt y,z] + [y,x \bt z] \label{eq:infpostlie:1}\\
[x,y] \bt z &= \R(x,y,z) - \R(y,x,z) - \L(x,y,z) + \L(y,x,z) \label{eq:infpostlie:2}
\end{align}
for all $x,y,z \in \g$, where 
\begin{align*}
\R(x,y,z) &= x \wt (y \bt z) + x \bt (y \wt z)\\
\L(x,y,z) &= (x \wt y) \bt z + (x \bt y) \wt z.
\end{align*}

A morphism between two infinitesimal post-Lie algebras $(\g,[\cdot,\cdot], \wt,\bt)$ and  $(\g',[\cdot,\cdot]', \wt',\bt')$ is a morphism of post-Lie algebras $f: (\g,[\cdot,\cdot],\wt) \to (\g',[\cdot,\cdot]',\wt')$ satisfying $f(x \bt y) = f(x) \bt' f(y)$. 

We shall denote the category of infinitesimal post-Lie algebras by $\mathrm{IPLie}$.
\end{definition}

It is easy to see that, for a fixed post-Lie algebra $(\g, [\cdot,\cdot], \wt)$, the space of all compatible infinitesimal post-Lie structures forms a vector space. Hence, in particular any post-Lie algebra $(\g,[\cdot,\cdot],\wt)$ is infinitesimal post-Lie with $\bt=0$.

\begin{remark}\label{def-inf-deformation}
In \cite[Definition 3.12]{LazST}, A. Lazarev, Y. Sheng and R. Tang introduced the notion of \textit{infinitesimal deformation} of a post-Lie algebra $(\g,[\cdot,\cdot],\wt)$ consisting of a pair of bilinear maps
\begin{align*}
\{ \cdot,\cdot\} : \g \times \g \to \g \quad \text{and} \quad \bt: \g \times \g \to \g
\end{align*}
such that the maps
\begin{align*}
\llbracket \cdot,\cdot \rrbracket = [\cdot,\cdot] + \hbar \{\cdot,\cdot\}  \quad \text{and} \quad  \widetilde{\wt} = \wt + \hbar \bt 
\end{align*}
make the triple $( R \ten \g, \llbracket \cdot,\cdot \rrbracket , \widetilde{\wt}) $ a post-Lie algebra in the category of $R$-modules, where $R= \Bbbk[\hbar]/(\hbar^2)$ and $\hbar$ is a formal parameter. Given a post-Lie algebra $(\g,[\cdot,\cdot],\wt)$, then $\omega:= (\{\cdot,\cdot\},\bt)$ is an infinitesimal deformation of $(\g,[\cdot,\cdot],\wt)$ if and only if $\omega$ is a $2$-cocycle in a certain cochain complex, see \cite[Theorem 3.13]{LazST}.

Note that Definition \ref{def:infinitesimalPostLie} can be understood as the special case of the above when setting $\{\cdot,\cdot\}=0$. This means that we are considering infinitesimal deformations with undeformed Lie brackets.
\end{remark}

For any infinitesimal post-Lie algebra $(\g,[\cdot,\cdot], \wt,\bt)$ and $x \in \g$, consider the linear map
\begin{align}
E_x :\g \to \g, \quad y \mapsto x \bt y.
\end{align}
Therefore, Equation \eqref{eq:infpostlie:1} is equivalent to the fact that $E_x$ belongs to $\Der(\g)$, while Equation \eqref{eq:infpostlie:2} is equivalent to
\begin{align}
\label{eq:interpretation-inf-post-Lie}
E_{[x,y]} = [D_x, E_y] - [D_y,E_x] - E_{x \wt y } + E_{y \wt x} - D_{x \bt y} + D_{y \bt x}.
\end{align}

\begin{remark}
\label{remark-simple-inf-post-lie}
Let $\g$ be a finite-dimensional complete Lie algebra over a field of characteristic zero. If $(\wt,\bt)$ is an infinitesimal post-Lie structure on $\g$, it follows that there exist linear maps $\varphi,\psi: \g \to \g$ such that 
\[ x \wt y = [\varphi(x),y] \quad \text{and} \quad x \bt y = [\psi(x),y],\]
and we can rewrite condition \eqref{eq:infpostlie:2} in terms of $\varphi$ and $\psi$ as follows:
\begin{align}
\label{eq:phi-psi-condition}
\psi([x,y]) = [\varphi(x),\psi(y)] - [\varphi(y), \psi(x)] - \psi\big([\varphi(x),y] - [\varphi(y),x] \big)  -\varphi\big([\psi(x),y] - [\psi(y),x] \big).
\end{align}
Note also that equation \eqref{eq:phi-psi-condition} can be rewritten in terms of the representation \eqref{eq:rho-phi} as
\[
\psi([x,y]_{\triangleright})
=
\rho_\varphi(x)\psi(y)
-
\rho_\varphi(y)\psi(x).
\]
This is precisely the Chevalley--Eilenberg
$1$-cocycle condition for $\psi$ with coefficients
in the $\g_\wt$-module $\g_{\rho_\varphi}$, where $\g_{\rho_\varphi}$ denotes $\g$ endowed
with the action $\rho_\varphi$.
Therefore, the vector space of infinitesimal
post-Lie structures compatible with $\wt$
is canonically isomorphic to
$Z^1(\g_\wt,\g_{\rho_\varphi})$.

\end{remark}
A direct application of Equation \eqref{eq:phi-psi-condition} will be the classification of all infinitesimal post-Lie structures on $\s\l(2)$, see next section. 
\begin{remark}\label{rmk:infinitesimalfortrivialpostLie}
Given a Lie algebra $(\g,[\cdot,\cdot])$, let us analyze the infinitesimal structures associated with the trivial post-Lie structures.
\begin{itemize}
    \item[1)] A bilinear map $\bt:\g\times\g\to\g$ provides an infinitesimal post-Lie structure on $(\g,[\cdot,\cdot],\wt=0)$ if (and only if) it satisfies
\begin{equation}\label{eq:infeqtrivialpostLie}
x \bt [y,z] = [x \bt y,z] + [y,x \bt z],\qquad [x,y]\bt z=0. 
\end{equation}
\item[2)] A bilinear map $\bt:\g\times\g\to\g$ provides an infinitesimal post-Lie structure on $(\g,[\cdot,\cdot],\wt=-[\cdot,\cdot])$ if (and only if) it satisfies \eqref{eq:infeqtrivialpostLie}.
Indeed, we have
\begin{align*}
    \R(x,y,z) - \R(y,x,z)&=x \wt (y \bt z) + x \bt (y \wt z)-y \wt (x \bt z)- y \bt (x \wt z)\\&=-[x,y\bt z]-x\bt[y,z]+[y,x\bt z]+y\bt[x,z]\\&\overset{\eqref{eq:infpostlie:1}}{=}-[x\bt y,z]+[y\bt x,z]
\end{align*}
and
\begin{align*}
    - \L(x,y,z) + \L(y,x,z)&=-(x \wt y) \bt z - (x \bt y) \wt z+(y \wt x) \bt z + (y \bt x) \wt z\\&=[x,y] \bt z +[x \bt y,z]-[y,x] \bt z -[y \bt x,z]\\&=2[x,y] \bt z +[x \bt y,z] -[y \bt x,z].
\end{align*}
\end{itemize}
Therefore, in both cases, if $[\g,\g]=\g$ then $\bt=0$. This happens, for instance, when $\g$ is simple.
\end{remark}

\subsubsection{Classification of infinitesimal post-Lie structures on $\s\l(2)$}
\label{subsection-inf-post-lie-sl2}
In Subsection \ref{section-post-lie-sl2} we discussed the classification (up to isomorphism) of post-Lie structures on $\s\l(2)$. By Remark \ref{rmk:infinitesimalfortrivialpostLie}, we know that the only infinitesimal post-Lie structure compatible with the trivial post-Lie structures on $\s\l(2)$ is $\bt=0$. In the following, we shall provide a multi-parameter family of infinitesimal post-Lie structures compatible with the non-trivial post-Lie algebra structures of $\s\l(2)$ mentioned above. Note that, in view of Remark \ref{remark-simple-inf-post-lie}, this amounts, for any post-Lie structure associated to a Rota-Baxter operator $\varphi$, to finding a linear map $\psi: \s\l(2) \to \s\l(2)$ satisfying Equation \eqref{eq:phi-psi-condition}. We shall denote the corresponding vector space of compatible infinitesimal post-Lie structures by $Z_\varphi$.
\begin{theorem}\label{thm:sl2-infinitesimal-classification}
Every infinitesimal post-Lie structure on $\s\l(2)$ is isomorphic to
one given by $x\wt y=[\varphi(x),y]$ and $x\bt y=[\psi(x),y]$ in the following
table. All displayed scalar parameters are arbitrary complex numbers.
\[
\renewcommand{\arraystretch}{1.45}
\begin{array}{c|c|c|c|c}
\varphi&\psi(e)&\psi(f)&\psi(h)&\dim Z_\varphi\\ \hline
0,\ -\mathrm{id}&0&0&0&0\\
\varphi_t,\ t\ne0,-1&u h&v h&
\dfrac{2t}{t+1}v e+\dfrac{2(t+1)}t u f+w h&3\\
\varphi_0&0&b e+k h&c f+l h&4\\
\varphi_{-1}&d f+u h&0&c e+l h&4\\
\varphi_*,\ \overline \varphi_*&0&p e-r f+q h&r(f+h)&3
\end{array}
\]
For fixed $\varphi$, two different $\psi,\psi'$ yield isomorphic infinitesimal
post-Lie algebras if and only if $\psi' \circ \chi=\chi \circ \psi$ for some
$\chi \in \Aut(\s\l(2))$ with $\chi \circ \varphi =\varphi \circ \chi$.
\end{theorem}
\begin{proof}
As mentioned above, the cases $0$ and $-\mathrm{id}$ follow from
Remark~\ref{rmk:infinitesimalfortrivialpostLie}. For $\varphi=\varphi_t$, write
\begin{equation}\label{eq:sl2-psi-on-generators}
\psi(e)=a_1e+a_2f+a_3h,\quad
\psi(f)=b_1e+b_2f+b_3h,\quad
\psi(h)=c_1e+c_2f+c_3h,
\end{equation}
where $a_i,b_i,c_i$ are complex parameters.
A lengthy but elementary computation shows that inserting the pairs $(e,f),(e,h),(f,h)$ in Equation \eqref{eq:phi-psi-condition} respectively gives
\begin{align*}
-\bigl((t+1)a_1+tb_2\bigr)h&=0,\\
2a_1e+4(t+1)a_2f+\bigl(2(t+1)a_3-tc_2\bigr)h&=0,\\
-4tb_1e+2b_2f+\bigl((t+1)c_1-2tb_3\bigr)h&=0,
\end{align*} 
which is equivalent to 
\begin{equation}\label{eq:sl2-coefficient-system}
\begin{gathered}
a_1=b_2=0,\qquad (t+1)a_2=tb_1=0,\\
2(t+1)a_3=tc_2,\qquad (t+1)c_1=2tb_3.
\end{gathered}
\end{equation}
Solving separately for $t\ne0,-1$, $t=0$ and $t=-1$ yields the claimed result. The same process for $\varphi_*, \overline \varphi_*$ gives the relations
\[
a_1=a_2=a_3=c_1=0,\qquad b_2=-c_3,\qquad c_2=c_3.
\]
Finally, since $\s\l(2)$ has zero center, a Lie algebra
automorphism preserves both post-Lie and infinitesimal post-Lie structure if and only if it intertwines
their respective operators $\varphi$ and $\psi$. 
\end{proof}
Although Theorem \ref{thm:sl2-infinitesimal-classification} classifies all infinitesimal
post-Lie structures on $\mathfrak{sl}(2)$, not all
of them extend to formal deformations, as shown
by the following result.
\begin{proposition}\label{prop:sl2-obstructed-deformation}
For the Rota-Baxter operator $\varphi_0$, the map
\[
\psi(e)=0,\qquad \psi(f)=be+kh,\qquad \psi(h)=cf+lh
\]
extends to a formal Rota--Baxter operator of weight $1$ on
$\s\l(2)[[\hbar]]$ if and only if $bl=0$.
\end{proposition}
\begin{proof}
Suppose that
$\varphi_\hbar=\varphi_0+\hbar\psi
+\mathcal{O}(\hbar^2)$
is a formal Rota--Baxter operator of weight $1$.
Evaluating
\[
[\varphi_\hbar(x),\varphi_\hbar(y)]
=\varphi_\hbar\bigl(
[x,y]+[\varphi_\hbar(x),y]
+[x,\varphi_\hbar(y)]\bigr)
\]
at $(x,y)=(f,h)$, the $e$-components of the
coefficients of $\hbar^2$ on the left- and
right-hand sides are respectively $-2bl$ and
$2bl$.
Hence $bl=0$.
Conversely, suppose that $bl=0$ and consider
the map
\[
\widehat{\varphi}_\hbar(e)=0,\qquad
\widehat{\varphi}_\hbar(f)=-f+\hbar be,\qquad
\widehat{\varphi}_\hbar(h)=\hbar lh.
\]
A direct computation shows that
$\widehat{\varphi}_\hbar$ is a
Rota--Baxter operator of weight $1$.
Finally, consider the element $u=-ke+\frac{c}{2}f$.
Then
\[
[\mathrm{ad}_u,\varphi_0](e)=0,\qquad
[\mathrm{ad}_u,\varphi_0](f)=kh,\qquad
[\mathrm{ad}_u,\varphi_0](h)=cf.
\]
Hence, $\varphi_\hbar
:=e^{\hbar\mathrm{ad}_u}
\circ\widehat{\varphi}_\hbar
\circ e^{-\hbar\mathrm{ad}_u}$
is a formal Rota--Baxter operator of weight $1$
whose first-order term is precisely $\psi$.
\end{proof}

\subsubsection{A geometric interpretation of infinitesimal post-Lie algebras}\label{sec:geometry}

Let $\nabla\colon\Gamma^\infty(TM)\times\Gamma^\infty(TM)\to\Gamma^\infty(TM)$ be a covariant derivative on a finite-dimensional smooth real manifold $M$. Assume that $\nabla$ is flat and that its torsion $\mathrm{Tor}^\nabla$ is covariantly constant. As described in Section \ref{sec:geompost}, this gives a post-Lie algebra $(\Gamma^\infty(TM),B,\nabla)$, where we set $B:=-\mathrm{Tor}^\nabla$. Its subadjacent Lie algebra is the Lie algebra $(\Gamma^\infty(TM),[\cdot,\cdot])$ of vector fields on $M$. We give a tensorial criterion for its geometric infinitesimal deformations, followed by examples and a rigidity result.

Given a tensor $A\in\Gamma^\infty(T^*M\otimes T^*M\otimes TM)$ we use the notation $A_X(Y)=A(X,Y)$ and regard $A$ as a one-form with values in $\operatorname{End}(TM)$. Then, the covariant exterior derivative of $A$ is defined by
\begin{equation}\label{eq:covariant-exterior-A}
    (d_\nabla A)(X,Y)Z
    =\nabla_X(A(Y,Z))-A(Y,\nabla_XZ)
    -\nabla_Y(A(X,Z))+A(X,\nabla_YZ)-A([X,Y],Z),
\end{equation}
where the connection on $\operatorname{End}(TM)$ is the one induced by $\nabla$. We say that $A$ is symmetric if $A(X,Y)=A(Y,X)$ holds for all $X,Y\in\Gamma^\infty(TM)$.

\begin{proposition}\label{prop:geometric-infinitesimal}
Let $(\Gamma^\infty(TM),B,\nabla)$ be the post-Lie algebra as above, and consider a symmetric $(1,2)$-tensor $A\in\Gamma^\infty(T^*M\otimes T^*M\otimes TM)$. Then, the following conditions are equivalent:
\begin{enumerate}[label=\textup{(\roman*)}]
\item $(\Gamma^\infty(TM),B,\nabla,\bt)$, with $X\bt Y=A(X,Y)$, is an
infinitesimal post-Lie algebra;
\item the connection $\widetilde\nabla=\nabla+\hbar A$ over
$\mathbb R[\hbar]/(\hbar^2)$ is flat and $\mathrm{Tor}^{\widetilde\nabla}=\mathrm{Tor}^\nabla$ is covariantly constant;
\item we have
\begin{align}
d_\nabla A&=0,\label{eq:infcon1}\\
A_X B(Y,Z)&=B(A_XY,Z)+B(Y,A_XZ)
\label{eq:infcon2}
\end{align}
for all $X,Y,Z\in\Gamma^\infty(TM)$.
\end{enumerate}
If one, and thus all, of the above conditions hold, the connections
$\nabla^t=\nabla+tA$, $t\in\mathbb R$, have torsion $\mathrm{Tor}^{\nabla^t}=\mathrm{Tor}^\nabla$ and satisfy
$\nabla^t\mathrm{Tor}^\nabla=0$. Moreover, they are flat for every $t$ if and only if
\begin{equation}\label{eq:geometric-linear-integrability}
[A_X,A_Y]=0\qquad\text{for all }X,Y\in\Gamma^\infty(TM).
\end{equation}
\end{proposition}

\begin{proof}
Let us start by considering the 1-parameter family $\nabla^t=\nabla+tA$ with $t\in\mathbb{R}$. By the tensoriality of $A$ we have that $\nabla^t$ is a covariant derivative and by the symmetry of $A$ it follows that
$$
\mathrm{Tor}^{\nabla^t}(X,Y)
=\nabla_XY+tA(X,Y)-\nabla_YX-tA(Y,X)-[X,Y]
=\mathrm{Tor}^\nabla(X,Y).
$$ 
Moreover, by using $\mathrm R^\nabla=0$ and $\nabla\mathrm{Tor}^\nabla=0$, one verifies that the curvature of $\nabla^t$ and the covariant derivative of its torsion are given by
\begin{align}
    \mathrm R^{\nabla^t}(X,Y)Z
    &=t(d_\nabla A)(X,Y)Z+t^2[A_X,A_Y]Z,\label{CurvNablaT}\\
    (\nabla^t_X\mathrm{Tor}^\nabla)(Y,Z)
    &=t\bigl(A_X\mathrm{Tor}^\nabla(Y,Z)-\mathrm{Tor}^\nabla(A_XY,Z)-\mathrm{Tor}^\nabla(Y,A_XZ)\bigr).
\end{align}
Repeating the above computations for the covariant derivative $\widetilde\nabla=\nabla+\hbar A$ over $\mathbb{R}[\hbar]/(\hbar^2)$, we infer that $\mathrm R^{\widetilde\nabla}=0$ is equivalent to the condition $d_\nabla A=0$ and that $\widetilde\nabla(\mathrm{Tor}^\nabla)=0$ is equivalent to $A_X\mathrm{Tor}^\nabla(Y,Z)=\mathrm{Tor}^\nabla(A_XY,Z)+\mathrm{Tor}^\nabla(Y,A_XZ)$ for all $X,Y,Z\in\Gamma^\infty(TM)$. This proves the equivalence of (ii) and (iii). For the equivalence of (i) and (iii), one observes that the first infinitesimal post-Lie identity
\eqref{eq:infpostlie:1} is precisely \eqref{eq:infcon2}, while the second infinitesimal post-Lie identity \eqref{eq:infpostlie:2} reduces to \eqref{eq:infcon1} after making use of the symmetry of $A$ and substituting
$B(X,Y)=[X,Y]-\nabla_XY+\nabla_YX$. 

For the $1$-parameter family $\nabla^t$, the additional assumption \eqref{eq:geometric-linear-integrability} is then equivalent to flatness of $\nabla^t$ because of the equality \eqref{CurvNablaT}.
\end{proof}

Note that \eqref{eq:infcon2} says that $A_X$ is a derivation of the
Lie algebra $(\Gamma^\infty(TM),B)$. 

\medskip

Consider now a Lie group $G$ endowed with the canonical connection $\nabla\colon\Gamma^\infty(TG)\times\Gamma^\infty(TG)\to\Gamma^\infty(TG)$ which vanishes on left-invariant vector fields and recall that its torsion is $\mathrm{Tor}^\nabla(X,Y)=-[X,Y]$ on left-invariant vector fields $X,Y\in\Gamma^\infty(TG)^L$. Left-invariant tensors on $G$ are completely determined by their values on left-invariant vector fields. By identifying the latter with elements of the Lie algebra $\mathfrak{g}=\mathrm{Lie}(G)$, it follows that a left-invariant symmetric tensor $A\in\Gamma^\infty(T^*G\otimes T^*G\otimes TG)$ is completely determined by the symmetric bilinear map $a=A_{e}\colon\g\times\g\to\g$ via
\begin{equation}
    A(x^\#,y^\#)=a(x,y)^\#
\end{equation}
for $x,y\in\mathfrak{g}$, where $x^\#\in\Gamma^\infty(TG)$ denotes the fundamental (left-invariant) vector field associated to $x$.
\begin{corollary}
Let $G$ be a Lie group endowed with its canonical left-invariant flat covariant derivative $\nabla$. Consider a left-invariant symmetric $(1,2)$-tensor $A$ with corresponding symmetric bilinear map $a\colon\mathfrak{g}\times\mathfrak{g}\to\mathfrak{g}$. Then $(\Gamma^\infty(TG),-\mathrm{Tor}^\nabla,\nabla,\bt)$ with $X\bt Y=A(X,Y)$ is an infinitesimal post-Lie algebra if and only if
\begin{equation}\label{eq:parallel-geometric-conditions}
a([x,y],z)=0,\qquad
a(x,[y,z])=[a(x,y),z]+[y,a(x,z)]
\end{equation}
hold for all $x,y,z\in\mathfrak{g}$.
\end{corollary}
\begin{proof}
The covariant derivative $\nabla$ is flat and its torsion is covariantly constant, as one easily verifies on left-invariant vector fields. This implies that $\Gamma^\infty(TG)$ is a post-Lie algebra with Lie bracket $-\mathrm{Tor}^\nabla$ and post structure $\nabla$. Note that on left-invariant vector fields $-\mathrm{Tor}^\nabla(x^\#,y^\#)=[x,y]$. Evaluating \eqref{eq:covariant-exterior-A} on the fundamental vector fields of $x,y,z\in\mathfrak{g}$ gives $(d_\nabla A)(x^\#,y^\#)z^\#=-a([x,y],z)^\#$ and thus the conditions on $A$ in Proposition \ref{prop:geometric-infinitesimal} reduce to \eqref{eq:parallel-geometric-conditions}.
\end{proof}

\begin{example}[Heisenberg Lie algebra]
Let $(V,\omega)$ be a real symplectic vector space of dimension $2n$,
$n\geq1$, and endow $\mathfrak h_{2n+1}=V\oplus\mathbb Rz$ with the Lie bracket
\[
[u,v]=\omega(u,v)z,\qquad [z,\mathfrak h_{2n+1}]=0.
\]
For the canonical left-invariant flat covariant derivative $\nabla$ on its simply connected Lie group $G$, the
space of symmetric maps $a\colon\mathfrak{h}_{2n+1}\times\mathfrak{h}_{2n+1}\to\mathfrak{h}_{2n+1}$ satisfying
\eqref{eq:parallel-geometric-conditions} is parametrized by
\[
\operatorname{Sym}^3(V^*)\oplus\operatorname{Sym}^2(V^*).
\]
More explicitly, a pair $(c,Q)$ determines the tensor whose value on
left-invariant fields is
\begin{equation}\label{eq:heisenberg-cubic-parametrization}
a(z,\cdot)=0,\qquad a(u,v)=P(u,v)+Q(u,v)z,
\qquad \omega(P(u,v),w)=c(u,v,w).
\end{equation}
Consequently, this space has dimension $\binom{2n+2}{3}+\binom{2n+1}{2}$. In fact, since $[\mathfrak h_{2n+1},\mathfrak h_{2n+1}]=\mathbb Rz$, the first condition in \eqref{eq:parallel-geometric-conditions} gives $a(z,\cdot)=0$ and $a(\cdot,z)=a(z,\cdot)=0$ by symmetry. Write $a(u,v)=P(u,v)+Q(u,v)z$, with $P$ and $Q$ symmetric. The remaining condition, evaluated on $x,u,v\in V$, is $\omega(P(x,u),v)+\omega(u,P(x,v))=0$. Thus $c(x,u,v):=\omega(P(x,u),v)$ is symmetric in its last two arguments, and symmetry of $P$ makes it symmetric in its first two arguments, as well. Hence $c$ is fully symmetric. Conversely, nondegeneracy of $\omega$ recovers $P$ uniquely from any symmetric tensor $c$, and the displayed identity follows from that symmetry. The central part $Q$ is arbitrary.

To give an explicit example, let us consider $\mathfrak h_3=V\oplus\mathbb{R}z$ with basis $e_1,e_2,e_3=z$ and with symplectic structure on $V$ determined by $\omega(e_1,e_2)=1$, thus with Lie bracket on $\mathfrak h_3$ given by $[e_1,e_2]=e_3$. Following the
previous reasoning, there is a seven-parameter family
\[
\begin{aligned}
a(e_1,e_1)&=\alpha e_1+\beta e_2+\gamma e_3,\\
a(e_1,e_2)&=\delta e_1-\alpha e_2+\eta e_3,\\
a(e_2,e_2)&=\zeta e_1-\delta e_2+\theta e_3,
\qquad a(e_3,\cdot)=0,
\end{aligned}
\]
where all seven parameters are arbitrary real numbers. The subfamily $c=0$
in \eqref{eq:heisenberg-cubic-parametrization} satisfies
$A_XA_Y=0$, thus also the additional condition \eqref{eq:geometric-linear-integrability} and therefore $\nabla+tA$ is flat and has covariantly constant torsion for every $t$.
\end{example}

The preceding flexibility contrasts with rigidity when the torsion Lie
algebras are perfect.

\begin{proposition}\label{prop:perfect-torsion-rigidity}
Let $\nabla$ be a flat connection on a finite-dimensional smooth real manifold $M$ with covariantly constant torsion $\mathrm{Tor}^\nabla$, and
suppose that the Lie algebra $(T_pM,-\mathrm{Tor}^\nabla_p)$ is
perfect for all $p\in M$. Then, the only infinitesimal post-Lie algebra structure on $(\Gamma^\infty(TM),-\mathrm{Tor}^\nabla,\nabla)$ corresponding to a symmetric $(1,2)$-tensor is $\bt=0$.
\end{proposition}

\begin{proof}
Assume that $(\Gamma^\infty(TM),-\mathrm{Tor}^\nabla,\nabla,\bt)$ is an infinitesimal post-Lie algebra with $X\bt Y=A(X,Y)$ for a symmetric $(1,2)$-tensor $A$. Then \eqref{eq:infcon2} and the symmetry of $A$ imply that $A_p$ is a symmetric biderivation of $(T_pM,-\mathrm{Tor}^\nabla_p)$ for every $p$, i.e., a symmetric bilinear map that is a derivation in each argument. By \cite[Theorem~5.1]{BBO25}, every symmetric biderivation of a finite-dimensional perfect Lie algebra over a field of characteristic zero vanishes. Hence $\bt=0$.
\end{proof}

\begin{remark}
If in Proposition \ref{prop:perfect-torsion-rigidity} we replace the assumption that $(T_pM,-\mathrm{Tor}^\nabla_p)$ is perfect by the stronger assumption that it is semisimple, the proof becomes very elementary. Indeed, if $\kappa$ is the nondegenerate Killing form, each derivation is inner and is skew-symmetric with respect to $\kappa$. Consequently, $C_p(x,y,z):=\kappa(A_p(x,y),z)$ is symmetric in its first two arguments and antisymmetric in its last two arguments, which give
\[
C_p(x,y,z)=C_p(y,x,z)=-C_p(y,z,x)=-C_p(z,y,x)
=C_p(z,x,y)=C_p(x,z,y)=-C_p(x,y,z)
\]
for all $x,y,z\in T_pM$. Therefore $C_p=0$, and nondegeneracy of $\kappa$ implies $A_p=0$. For example, the canonical flat connection on $SU(2)\cong S^3$ has nonzero parallel torsion with fibres isomorphic to $\mathfrak{su}(2)$, and hence admits no nonzero symmetric geometric infinitesimal deformation with that torsion bracket fixed.
\end{remark}

\subsection{Infinitesimal post-Hopf algebras}\label{sec:IPH}

\subsubsection{Definition and main properties}

We now introduce infinitesimal deformations of post-Hopf algebras $(H,\wp)$, where we keep the underlying Hopf algebra structure undeformed and consider first-order deformations of $\wp$. This will turn out to be consistent with the infinitesimal deformations of post-Lie algebras which we discussed in Section \ref{sec:IPL}, i.e., with deformations of post-Lie algebras, where we keep the Lie bracket undeformed.

\begin{definition}\label{def:infinitesimalpostHopf}
An \textbf{infinitesimal post-Hopf algebra} is a triple $(H,\wp, \bp)$, where $(H,\wp)$ is a post-Hopf algebra and $\bp:H\ot H\to H$ is a linear map such that the following equalities hold:
\begin{eqnarray}
\Delta(a\bp b)&=&(a_{1}\wp b_{1})\ot(a_{2}\bp b_{2})+(a_{1}\bp b_{1})\ot(a_{2}\wp b_{2})\label{deltalinear}\\
a\bp(bc)&=&%
    (a_{1}\bp b)(a_{2}\wp c)+(a_{1}\wp b)(a_{2}\bp c)\label{multiplicativity}\\
a\bp(b\wp c)+a\wp(b\bp c)&=& (a_{1}(a_{2}\bp b))\wp c+(a_{1}(a_{2}\wp b))\bp c \label{deformedassociativity}
\end{eqnarray}
for all $a,b,c \in H$.

A morphism between two infinitesimal post-Hopf algebras $(H,\wp,\bp)$ and $(H',\wp',\bp')$ is a morphism of post-Hopf algebras $f:(H,\wp)\to(H',\wp')$ such that $f(a\bp b)=f(a)\bp'f(b)$, for all $a,b\in H$.

We shall denote the category of infinitesimal post-Hopf algebras by $\mathrm{IPHopf}$. 
\end{definition}

Clearly, any post-Hopf algebra $(H,\wp)$ is an infinitesimal post-Hopf algebra with $\bp=0$. We also observe that if we choose $\bp=\wp$, from \eqref{wpproduct} and \eqref{multiplicativity} we get $(a_{1}\wp b)(a_{2}\wp c)=0$ for every $a,b,c\in H$. Hence, choosing $c=1$, we get $a\wp b=0$ for all $a,b\in H$ since $\wp$ is a morphism of coalgebras. Similarly, choosing $\bp=\varepsilon\ot\mathrm{id}$ (or $\mathrm{id}\ot\varepsilon$) from \eqref{deltalinear} we immediately get $a\wp b=0$. In fact, as the following result shows, $\bp$ is very different from an action.

\begin{lemma}\label{lem:propertiesinfpostHopf}
Let $(H,\wp,\bp)$ be an infinitesimal post-Hopf algebra. Then for all $a,b \in H$ we have
\[1\bp a=0, \qquad a\bp 1=0, \qquad \text{and} \qquad  \varepsilon(a\bp b)=0  .\]
\end{lemma}

\begin{proof}
First, since $1\wp a=a$ for all $a\in H$, we get
\[
1\bp1\overset{\eqref{multiplicativity}}{=}(1\bp 1)(1\wp1)+(1\wp1)(1\bp 1)=1\bp 1+1\bp 1,
\]
so $1\bp 1=0$. Therefore, we get
\[
\begin{split}
1\bp a+1\bp a&=1\bp(1\wp a)+1\wp(1\bp a)\overset{\eqref{deformedassociativity}}{=}(1(1\bp1))\wp a+(1(1\wp1))\bp a\\&=(1\bp1)\wp a+1\bp a=1\bp a
\end{split}
\]
and hence $1\bp a=0$. Moreover, since $a\wp1=\varepsilon(a)1$, we obtain
\[
a\bp 1\overset{\eqref{multiplicativity}}{=}(a_{1}\bp 1)(a_{2}\wp1)+(a_{1}\wp1)(a_{2}\bp1)=a\bp1+a\bp1,
\]
hence $a\bp1=0$. Finally, we compute
\[
\begin{split}
a\bp b&=(\mathrm{id}\ot\varepsilon)\Delta(a\bp b)\overset{\eqref{deltalinear}}{=}(a_{1}\wp b_{1})\varepsilon(a_{2}\bp b_{2})+(a_{1}\bp b_{1})\varepsilon(a_{2}\wp b_{2})=(a_{1}\wp b_{1})\varepsilon(a_{2}\bp b_{2})+a\bp b,
\end{split}
\]
hence $(a_{1}\wp b_{1})\varepsilon(a_{2}\bp b_{2})=0$. As a consequence, $\varepsilon(a\bp b)=0$. 
\end{proof}

The category $\mathrm{IPHopf}$ inherits the symmetric monoidal structure of $\mathrm{PHopf}$, which has been described in Proposition \ref{prop:PHopfbraided}.

\begin{proposition}
The category $\mathrm{IPHopf}$ is symmetric monoidal.
\end{proposition}
\begin{proof}
Let $(H,\wp,\bp)$ and $(H',\wp',\bp')$ be two infinitesimal post-Hopf algebras. By Proposition \ref{prop:PHopfbraided}, we already know that $(H\ot H',\wp_{\ot})$ is a post-Hopf algebra where $\wp_{\ot}:=(\wp\ot\wp')(\mathrm{id}_{H}\ot\tau_{H',H}\ot\mathrm{id}_{H'})$.
We define $\bp_{\ot}:(H\ot H')\ot(H\ot H')\to H\ot H'$ as follows:
\[
(a\ot b)\bp_{\ot}(c\ot d):=(a\bp c)\ot(b\wp' d)+(a\wp c)\ot(b\bp' d),\qquad \text{for any } a,c\in H,b,d\in H'.
\]
Since $\wp$ and $\wp'$ are morphisms of coalgebras and the linear maps $\bp$ and $\bp'$ satisfy \eqref{deltalinear}, one can easily show that $\bp_{\ot}$ also satisfies \eqref{deltalinear}. 
Moreover, since $\wp$ and $\wp'$ satisfy \eqref{wpproduct} and $\bp$ and $\bp'$ satisfy \eqref{multiplicativity}, $\bp_{\ot}$ satisfies \eqref{multiplicativity} as well.
Finally, using that $\wp$ and $\wp'$ satisfy \eqref{wpdefassoci} and $\bp$ and $\bp'$ satisfy \eqref{deformedassociativity}, one can prove that $\bp_{\ot}$ satisfies \eqref{deformedassociativity}.
Therefore, $(H\ot H',\wp_{\ot},\bp_{\ot})$ is an infinitesimal post-Hopf algebra. Moreover, $(\Bbbk,\mu_{\Bbbk},0)$ is an infinitesimal post-Hopf algebra and $(\mathrm{IPHopf},\otimes,\Bbbk)$ becomes a monoidal category with the same constraints of $(\mathrm{Vec}_{\Bbbk},\ot,\Bbbk)$. 
Finally, denoting by $\bp'_{\ot}$ the morphism of the infinitesimal post-Hopf algebra $H'\ot H$, we have
\[
\tau((a\ot b)\bp_{\ot}(c\ot d))=(b\wp' d)\ot(a\bp c)+(b\bp' d)\ot(a\wp c)=\tau(a\ot b)\bp'_{\ot}\tau(c\ot d),
\]
hence $\tau$ is a morphism in $\mathrm{IPHopf}$. 
\end{proof}

Infinitesimal post-Hopf algebras behave well under images through Hopf algebra maps, as the following result shows.
\begin{proposition}\label{prop:iph-image}
Let $(H,\wp,\bp)$ be an infinitesimal post-Hopf algebra and let $f:H\to H'$ be a Hopf algebra map. Set $I:=\ker f$ and $\gamma(a\ot b):=\beta_{\wp}(a)(b)$, for $a,b\in H$,
where $\beta_{\wp}$ is the convolution inverse of $\alpha_{\wp}:a\mapsto a\wp(-)$. Suppose that
\begin{equation}\label{eq:iph-image-stability}
p(I\ot H+H\ot I)\subseteq I
\qquad\text{for }p\in\{\wp,\bp,\gamma\}.
\end{equation}
Then $f(H)$ has a unique infinitesimal post-Hopf structure for which $f:H\to f(H)$ is a morphism in $\mathrm{IPHopf}$, given by
\[
f(a)\wp_f f(b)=f(a\wp b),\qquad
f(a)\bp_f f(b)=f(a\bp b).
\]
The convolution inverse of $\alpha_{\wp_f}$ is given by
\[
\beta_{\wp_f}(f(a))(f(b))=f\bigl(\beta_{\wp}(a)(b)\bigr).
\]
\end{proposition}
\begin{proof}
The assumptions for $\wp$, $\bp$ and $\gamma$ ensure that $\wp_{f}$, $\bp_{f}$ and $\beta_{\wp_f}$ are well-defined. 
Since $f:H\to f(H)$ is a surjective Hopf algebra map, applying $f$ or $f\ot f$ to the defining identities of the post-Hopf and infinitesimal post-Hopf structures of $H$ proves the same identities for the induced operations. 
Applying $f$ to
\[
\alpha_{\wp}(a_1)\bigl(\beta_{\wp}(a_2)(b)\bigr)
=\varepsilon(a)b
=\beta_{\wp}(a_1)\bigl(\alpha_{\wp}(a_2)(b)\bigr)
\]
shows that it is the required two-sided convolution inverse. Uniqueness follows from the surjectivity of $f$.
\end{proof}
As a consequence, we get the following result:
\begin{corollary}\label{cor:iph-quotient}
Let $(H,\wp,\bp)$ be an infinitesimal post-Hopf algebra and let $I$ be a Hopf ideal satisfying \eqref{eq:iph-image-stability}. Then $H/I$ inherits an infinitesimal post-Hopf structure, and the canonical projection is a morphism of infinitesimal post-Hopf algebras.
\end{corollary}
\begin{proof}
Apply Proposition \ref{prop:iph-image} to the canonical projection.
\end{proof}

As recalled in Remark \ref{rmk:subadj}, given a cocommutative post-Hopf algebra, one can define the so-called subadjacent Hopf algebra. The latter is equipped with a Hochschild 2-cocycle once we start with an infinitesimal post-Hopf algebra, as shown by the following result.

\begin{theorem}\label{thm:Hochschild}
    Let $(H,\wp,\bp)$ be a cocommutative infinitesimal post-Hopf algebra. Then, the subadjacent Hopf algebra $H_{\wp}:=(H,\star_{\wp},\eta,\Delta,\varepsilon,S_{\wp})$, where $a\star_{\wp} b:=a_{1}(a_{2}\wp b)$ and $S_{\wp}(a):=\beta_{\wp}a_{1}(S(a_{2}))$, is equipped with a Hochschild 2-cocycle which is defined by $a\star_{\bp}b:=a_{1}(a_{2}\bp b)$.
\end{theorem}

\begin{proof}
Setting $a\star_{\bp}b:=a_{1}(a_{2}\bp b)$, Equation \eqref{deformedassociativity} is as follows:
\begin{equation}\label{defassoc2}
    a\bp(b\wp c)+a\wp(b\bp c)= (a\star_{\bp}b)\wp c+(a\star_{\wp}b)\bp c. 
\end{equation}
We compute
\begin{align}\label{deltastarbullet}
\Delta(a\star_{\bp}b)&=\Delta(a_{1}(a_{2}\bp b))=a_{1}(a_{3}\bp b)_{1}\ot a_{2}(a_{3}\bp b)_{2}\\&\overset{\eqref{deltalinear}}{=}\nonumber a_{1}(a_{3}\bp b_{1})\ot a_{2}(a_{4}\wp b_{2})+a_{1}(a_{3}\wp b_{1})\ot a_{2}(a_{4}\bp b_{2})\\&\nonumber\overset{(!)}{=}a_{1}(a_{2}\bp b_{1})\ot a_{3}(a_{4}\wp b_{2})+a_{1}(a_{2}\wp b_{1})\ot a_{3}(a_{4}\bp b_{2})\\&\nonumber=(a_{1}\star_{\bp} b_{1})\ot (a_{2}\star_{\wp} b_{2})+(a_{1}\star_{\wp} b_{1})\ot(a_{2}\star_{\bp} b_{2}),
\end{align}
where in $(!)$ we used that $H$ is cocommutative. Therefore, we obtain
\[
\begin{split}
  a\star_{\wp}(b\star_{\bp}c)&+a\star_{\bp}(b\star_{\wp}c)\\&
  = a_{1}(a_{2}\wp(b_{1}(b_{2}\bp c)))+a_{1}(a_{2}\bp(b_{1}(b_{2}\wp c)))\\&\overset{\eqref{wpproduct},\eqref{multiplicativity}}{=}\underbrace{a_{1}(a_{2}\wp b_{1})(a_{3}\wp(b_{2}\bp c))}+a_{1}(a_{2}\bp b_{1})(a_{3}\wp(b_{2}\wp c))+\underbrace{a_{1}(a_{2}\wp b_{1})(a_{3}\bp(b_{2}\wp c))}\\&\overset{\eqref{defassoc2}}{=}a_{1}(a_{2}\wp b_{1})((a_{3}\star_{\bp} b_{2})\wp c)+a_{1}(a_{2}\wp b_{1})((a_{3}\star_{\wp} b_{2})\bp c)+a_{1}(a_{2}\bp b_{1})(a_{3}\wp(b_{2}\wp c))\\&\overset{\eqref{wpdefassoci}}{=}\underbrace{(a_{1}\star_{\wp} b_{1})((a_{2}\star_{\bp} b_{2})\wp c)}+(a_{1}\star_{\wp} b_{1})((a_{2}\star_{\wp} b_{2})\bp c)+\underbrace{(a_{1}\star_{\bp}b_{1})((a_{2}\star_{\wp} b_{2})\wp c)}\\&\overset{\eqref{deltastarbullet}}{=}(a\star_{\bp} b)_{1}((a\star_{\bp} b)_{2}\wp c)+(a\star_{\wp} b)_{1}((a\star_{\wp} b)_{2}\bp c)\\&=(a\star_{\bp}b)\star_{\wp}c+(a\star_{\wp}b)\star_{\bp}c,
\end{split}
\]
i.e.
\[
a\star_{\wp}(b\star_{\bp}c)-(a\star_{\wp}b)\star_{\bp}c+a\star_{\bp}(b\star_{\wp}c)-(a\star_{\bp}b)\star_{\wp}c=0.
\]
The latter equality means that $\star_{\bp}$ is a Hochschild 2-cocycle for the Hopf algebra $H_{\wp}$.
\end{proof}

\begin{remark}\label{rmk:iph-fixed-coalgebra-deformation}
Theorem \ref{thm:Hochschild} in fact provides an infinitesimal Hopf algebra deformation with unchanged coproduct. More precisely, put $R=\Bbbk[\hbar]/(\hbar^2)$ and $\mu_\hbar(a,b)=a\star_{\wp}b+\hbar(a\star_{\bp}b).$
Then $R\ot_{\Bbbk}H$, equipped with $\mu_\hbar$ and the scalar extensions of $\eta,\Delta,\varepsilon$, is a Hopf monoid in the category of $R$-modules. Indeed, the Hochschild cocycle identity proves associativity, while \eqref{deltastarbullet} proves multiplicativity of the coproduct. The unit and counit identities follow from $1\bp a=a\bp1=0$ and $\varepsilon(a\bp b)=0$. Finally, an antipode exists because the identity map is convolution-invertible modulo the square-zero ideal $\hbar\operatorname{End}_R(R\ot H)$. 
\end{remark}

We end this subsection by mentioning the following result, which is immediate from Definition \ref{def:infinitesimalpostHopf}; this will be used in the next subsection to classify the infinitesimal post-Hopf structures on Sweedler's Hopf algebra. 

\begin{lemma}
Let $H$ be a Hopf algebra and consider the post-Hopf algebra $(H,\wp:=\varepsilon\ot\mathrm{id})$. Given a linear map $\bp:H\ot H\to H$, we have that $(H,\wp,\bp)$ is an infinitesimal post-Hopf algebra if and only if
\begin{align}
\Delta(a\bp b)&=(a\bp b_1)\otimes b_2+b_1\otimes(a\bp b_2), \label{eq:lemma-iph-1}\\ 
(ab)\bp c&=\varepsilon(b)a\bp c+\varepsilon(a)b\bp c, \label{eq:lemma-iph-2}\\
a\bp(bc)&=(a\bp b)c+b(a\bp c). \label{eq:lemma-iph-3}
\end{align}
\end{lemma}

Note that Equation \eqref{eq:lemma-iph-2} means that $(-)\bp c:H\to H$ is a derivation of $H$ for any $c\in H$, considering $H$ as an $H$-bimodule with trivial actions, while Equation \eqref{eq:lemma-iph-3} means that $a\bp(-):H\to H$ is a derivation of $H$ for any $a\in H$, considering $H$ as an $H$-bimodule with its multiplication.

Next, for a Hopf algebra $H$, set $H^+=\ker\varepsilon$, $Q(H)=H^+/(H^+)^2$, and recall the vector space of Hopf biderivations
\[
\operatorname{BiDer}(H)
=\{\partial\in\operatorname{End}_{\Bbbk}(H):
\partial(ab)=\partial(a)b+a\partial(b),\quad
\Delta \circ \partial=(\partial\ot\mathrm{id}+\mathrm{id}\ot \partial)\circ \Delta\}.
\]

\begin{proposition}
\label{prop:iph-trivial-classification}
For any Hopf algebra $H$, the vector space of
infinitesimal post-Hopf structures on
$(H,\varepsilon\otimes\mathrm{id})$ is canonically
isomorphic to $\mathrm{Hom}_{\Bbbk}(Q(H),\mathrm{BiDer}(H))$.
More precisely, a linear map
$\Theta:Q(H)\to\mathrm{BiDer}(H)$ corresponds to
the infinitesimal post-Hopf structure
\begin{equation}
\label{eq:wp-bider}
a\bp b
=\Theta\bigl((a-\varepsilon(a)1)+(H^+)^2\bigr)(b).
\end{equation}
\end{proposition}

\begin{proof}
Given an infinitesimal post-Hopf structure
$\bp$, define
$\partial_a(b):=a\bp b$.
By \eqref{eq:lemma-iph-1} and \eqref{eq:lemma-iph-3},
$\partial_a\in\mathrm{BiDer}(H)$ for every $a\in H$.
Moreover, \eqref{eq:lemma-iph-2} gives $\partial_{ab}
=\varepsilon(a)\partial_b+\varepsilon(b)\partial_a.$
In particular, $\partial_1=0$ and
$\partial_{ab}=0$ for all $a,b\in H^+$.
Hence the restriction of $a\mapsto\partial_a$
to $H^+$ induces a unique linear map
\[
\Theta:Q(H)\longrightarrow\mathrm{BiDer}(H),
\qquad
\Theta\bigl(a+(H^+)^2\bigr)=\partial_a,
\quad a\in H^+.
\]
Conversely, given such a map $\Theta$, define
$\bp$ by Equation \eqref{eq:wp-bider}.
Equations \eqref{eq:lemma-iph-1},\eqref{eq:lemma-iph-3} follow from the fact
that $\Theta$ takes values in $\mathrm{BiDer}(H)$. Moreover, for $a,b\in H$, we have
\[
\begin{aligned}
ab-\varepsilon(ab)1
=(a-\varepsilon(a)1)(b-\varepsilon(b)1)
+\varepsilon(a)(b-\varepsilon(b)1)
+\varepsilon(b)(a-\varepsilon(a)1).
\end{aligned}
\]
Since the first term vanishes in $Q(H)$, applying
$\Theta$ and evaluating at $c\in H$ yields \eqref{eq:lemma-iph-2}.
Finally, the two constructions are mutually
inverse, since $\partial_1=0$ allows us to
recover $\bp$ from its restriction to $H^+$,
while Equation \eqref{eq:wp-bider} recovers
$\Theta$ upon restriction to $H^+$.
\end{proof}

\begin{corollary}\label{cor:iph-trivial-rigidity}
If $Q(H)=0$ or $\operatorname{BiDer}(H)=0$, the trivial post-Hopf structure on $H$ has no nonzero infinitesimal structures. In particular, this holds whenever $H$ is semisimple as an algebra.
\end{corollary}

\begin{corollary}\label{cor:iph-trivial-enveloping}
Suppose that $\operatorname{char}\Bbbk=0$. For every Lie algebra $\g$, the space of infinitesimal post-Hopf structures on $(\mathrm{U}(\g),\varepsilon\ot\mathrm{id})$ is isomorphic to $\operatorname{Hom}_{\Bbbk}\bigl(\g/[\g,\g],\Der(\g)\bigr).$
In particular, this structure is infinitesimally rigid when $\g$ is perfect.
\end{corollary}

\begin{proof}
It suffices to note that every Lie algebra satisfies
$Q(\mathrm{U}(\g)) \cong \g/[\g,\g]$ (see \cite{RU}),
while, in characteristic zero, restriction to primitive
elements induces an isomorphism
$\mathrm{BiDer}(\mathrm{U}(\g)) \cong \mathrm{Der}(\g)$.
Indeed, every derivation of $\g$ extends uniquely to
an algebra derivation of $\mathrm{U}(\g)$
(see \cite[Section 6.1]{Po}), which is also a
coderivation since the coderivation identity holds
on the primitive generators and extends to products
(see \cite[Section 2.4]{AS}).
The assertion follows from Proposition
\ref{prop:iph-trivial-classification}.
\end{proof}

For group algebras, rigidity holds for every post-Hopf structure, without any restriction on the characteristic, as explained in the following result.
\begin{proposition}\label{prop:iph-group-rigidity}
Let $G$ be a group and let $\wp$ be any post-Hopf structure on $\Bbbk G$. Then $(\Bbbk G,\wp)$ has no nonzero infinitesimal post-Hopf structures.
\end{proposition}
\begin{proof}
For $g,h\in G$, the element $u=g\wp h$ is grouplike because $\wp$ is a coalgebra map. Equation \eqref{deltalinear} gives
\[
\Delta(g\bp h)=u\ot(g\bp h)+(g\bp h)\ot u.
\]
Consequently $u^{-1}(g\bp h)$ is primitive. The group algebra has no nonzero primitive elements: comparison of coefficients in the basis $G\times G$ in $\Delta(v)=v\ot1+1\ot v$ first eliminates every coefficient at $g\ne1$, and then the coefficient at $1$. Thus $g\bp h=0$ for all $g,h\in G$, and bilinearity proves $\bp=0$.
\end{proof}

\subsubsection{Infinitesimal post-Hopf structures on Sweedler's Hopf algebra}\label{sec:Sweedler}
Let $\Bbbk$ be a field of characteristic different from 2. Consider Sweedler's Hopf algebra $H_4=\mathrm{span}_\Bbbk\{1,g,x,gx\}$ with $g^2=1$, $x^2=0$, $xg=-gx$, $\Delta(g)=g\otimes g$ and $\Delta(x)=x\otimes 1+g\otimes x$ as in 2) of Example \ref{ex:SweedlerpostHopf}. Endow $H_4$ with the trivial post-Hopf structure $a\wp b:=\varepsilon(a)b$. By evaluating \eqref{eq:lemma-iph-2} on $a=b=g$, we get,
\[
0=1\bp c=(gg)\bp c=\varepsilon(g)g\bp c+\varepsilon(g)g\bp c=2g\bp c,
\]
for any $c\in H_{4}$. Hence $g\bp c=0$, for all $c\in H_{4}$. We consider again \eqref{eq:lemma-iph-2} with $a=x$, $b=g$, obtaining
\[
(xg)\bp c=\varepsilon(g)x\bp c+\varepsilon(x)g\bp c=x\bp c, 
\]
while \eqref{eq:lemma-iph-2} with $a=g$, $b=gx$ gives us
\[
x\bp c=(ggx)\bp c=\varepsilon(gx)g\bp c+\varepsilon(g)gx\bp c=gx\bp c.
\]
Therefore, $x\bp c=gx\bp c=0$ for all $c\in H_{4}$. Therefore, $\bp=0$ is the unique infinitesimal post-Hopf structure on the post-Hopf algebra $(H_{4},\varepsilon\ot\mathrm{id})$. \medskip

Next, we classify infinitesimal post-Hopf structures on $H_{4}$, for the remaining post-Hopf structures on $H_{4}$. 
As recalled in 2) of Example \ref{ex:SweedlerpostHopf}, these are given by:
    \[
\begin{tabular}{c|cccc}
    $\wp_{0}$& $1$ & $g$ & $x$ & $gx$\\ \hline
    $1$ & $1$ & $g$ & $x$ & $gx$ \\ 
    $g$ & $1$ & $g$ & $-x$ & $-gx$ \\ 
    $x$ & $0$ & $0$ & $0$ & $0$ \\ 
    $gx$ & $0$ & $0$ & $0$ & $0$ \\ 
\end{tabular}\qquad
\begin{tabular}{c|cccc}
    $\wp_{1}$& $1$ & $g$ & $x$ & $gx$\\ \hline
    $1$ & $1$ & $g$ & $x$ & $gx$ \\ 
    $g$ & $1$ & $g$ & $-x$ & $-gx$ \\ 
    $x$ & $0$ & $0$ & $x$ & $gx$ \\ 
    $gx$ & $0$ & $0$ & $x$ & $gx$ \\ 
\end{tabular}
\]
The same computations work for both $\wp_{0}$ and $\wp_{1}$; we will simply write $\wp$. 

By Lemma \ref{lem:propertiesinfpostHopf}, we already know that $1\bp a=0=a\bp1$ for all $a\in H_{4}$. By applying \eqref{deltalinear} we get 
\[
\Delta(g\bp g)=(g\wp g)\ot(g\bp g)+(g\bp g)\ot(g\wp g)=g\ot(g\bp g)+(g\bp g)\ot g,
\]
hence $g\bp g=0$, since there are no non-zero $(g,g)$-primitive elements in $H_{4}$. As a consequence, 
\[
\Delta(x\bp g)=(x\wp g)\ot(1\bp g)+(g\wp g)\ot(x\bp g)+(x\bp g)\ot g+(g\bp g)\ot(x\wp g)=g\ot(x\bp g)+(x\bp g)\ot g,
\]
so that $x\bp g=0$ for the same reason. Moreover, we have
\begin{align*}
\Delta(g\bp x)&=(g\wp x)\ot(g\bp 1)+(g\wp g)\ot(g\bp x)+(g\bp x)\ot(g\wp 1)+(g\bp g)\ot(g\wp x)\\&=g\ot(g\bp x)+(g\bp x)\ot1, 
\\
\Delta(g\bp gx)&=(g\wp gx)\ot(g\bp g)+(g\wp 1)\ot(g\bp gx)+(g\bp gx)\ot(g\wp g)+(g\bp 1)\ot(g\wp gx)\\&=1\ot(g\bp gx)+(g\bp gx)\ot g,
\end{align*}
hence $g\bp x=a(1-g)+bx$, for $a,b\in\Bbbk$ and $g\bp gx=c(1-g)+dgx$, for $c,d\in\Bbbk$.

Moreover, we have
\begin{align*}
    \Delta(x\bp x)&=(x\wp x)\ot(1\bp 1)+(g\wp x)\ot(x\bp 1)+(x\wp g)\ot(1\bp x)+(g\wp g)\ot(x\bp x)\\&+(x\bp x)\ot 1+(x\bp g)\ot x+(g\bp x)\ot(x\wp 1)+(g\bp g)\ot(x\wp x)\\&=g\ot(x\bp x)+(x\bp x)\ot 1\\\Delta(x\bp gx)&=(g\wp gx)\ot(x\bp g)+(g\wp 1)\ot(x\bp gx)+(x\wp gx)\ot(1\bp g)+(x\wp 1)\ot(1\bp gx)\\&+(x\bp gx)\ot g+(x\bp 1)\ot gx+(g\bp gx)\ot(x\wp g)+(g\bp 1)\ot(x\wp gx)\\&=1\ot(x\bp gx)+(x\bp gx)\ot g,
\end{align*}
so that $x\bp x=a'(1-g)+b'x$ and $x\bp gx=c'(1-g)+d'gx$, for $a',b',c',d'\in\Bbbk$. Similarly, $gx\bp g=0$, $gx\bp x=a''(1-g)+b''x$ and $gx\bp gx=c''(1-g)+d''gx$.

By applying \eqref{multiplicativity}, we have
\begin{align*}
    c(1-g)+dgx&=g\bp(gx)=
    (g\bp g)(g\wp x)+(g\wp g)(g\bp x)=g(g\bp x)=-a(1-g)+bgx\\ 
    c'(1-g)+d'gx&=x\bp(gx)=
    (x\bp g)(1\wp x)+(g\bp g)(x\wp x)+(x\wp g)(1\bp x)+(g\wp g)(x\bp x)\\&=g(x\bp x)=-a'(1-g)+b'gx.
\end{align*}
Similarly, $c''(1-g)+d''gx=-a''(1-g)+b''gx$.
Therefore, we get
\[
\begin{tabular}{c|cccc}
    $\bp$& $1$ & $g$ & $x$ & $gx$\\ \hline
    $1$ & $0$ & $0$ & $0$ & $0$ \\ 
    $g$ & $0$ & $0$ & $a(1-g)+bx$ & $-a(1-g)+bgx$ \\ 
    $x$ & $0$ & $0$ & $a'(1-g)+b'x$ & $-a'(1-g)+b'gx$ \\ 
    $gx$ & $0$ & $0$ & $a''(1-g)+b''x$ & $-a''(1-g)+b''gx$ \\ 
\end{tabular}
\]
which implies
\[
\begin{split}
-2bx&=-a(1-g)-bx+a(1-g)-bx=-g\bp x+g\wp(a(1-g)+bx)=
g\bp(g\wp x)+g\wp(g\bp x)\\&\overset{\eqref{deformedassociativity}}{=} (g(g\bp g))\wp x+(g(g\wp g))\bp x=1\bp x=0
\end{split}
\]
so that $b=0$ and then 
\[
\begin{split}
g(g\bp x)&=(g\bp g)(g\wp x)+(g\wp g)(g\bp x)=g\bp(gx)=-g\bp(xg)\overset{\eqref{multiplicativity}}{=}-(g\bp x)(g\wp g)-(g\wp x)(g\bp g)\\&=-(g\bp x)g,
\end{split}
\]
i.e.\ $ag(1-g)=-a(1-g)g$, from which $a=0$. Therefore, we get
\[
\begin{tabular}{c|cccc}
    $\bp$& $1$ & $g$ & $x$ & $gx$\\ \hline
    $1$ & $0$ & $0$ & $0$ & $0$ \\ 
    $g$ & $0$ & $0$ & $0$ & $0$ \\ 
    $x$ & $0$ & $0$ & $a'(1-g)+b'x$ & $-a'(1-g)+b'gx$ \\ 
    $gx$ & $0$ & $0$ & $a''(1-g)+b''x$ & $-a''(1-g)+b''gx$\\ 
\end{tabular}
\]
Moreover, we have
\begin{align*}
x\bp(gx)&=(x\bp g)(1\wp x)+(g\bp g)(x\wp x)+(x\wp g)(1\bp x)+(g\wp g)(x\bp x)=g(x\bp x)=a'(g-1)+b'gx\\
x\bp(xg)&=(x\bp x)(1\wp g)+(g\bp x)(x\wp g)+(x\wp x)(1\bp g)+(g\wp x)(x\bp g)=(x\bp x)g=a'(g-1)+b'xg
\end{align*}
from which $a'=0$ follows. Similarly, one obtains $a''=0$. Finally, computing 
\[
-b'x=b'g\wp x=g\wp(x\bp x)=g\bp(x\wp x)+g\wp(x\bp x)= (g(g\bp x))\wp x+(g(g\wp x))\bp x=-gx\bp x=-b''x
\]
we get $b'=b''$. Therefore, $\bp$ must be of the form
\[
\begin{tabular}{c|cccc}
    $\bp$& $1$ & $g$ & $x$ & $gx$\\ \hline
    $1$ & $0$ & $0$ & $0$ & $0$ \\ 
    $g$ & $0$ & $0$ & $0$ & $0$ \\ 
    $x$ & $0$ & $0$ & $\lambda x$ & $\lambda gx$ \\ 
    $gx$ & $0$ & $0$ & $\lambda x$ & $\lambda gx$ \\ 
\end{tabular}
\]
for $\lambda\in\Bbbk$. One can check that the latter defines an exhaustive 1-parameter family of infinitesimal post-Hopf algebra structures for both the post-Hopf algebra structures $(H_{4},\wp_{0})$ and $(H_{4},\wp_{1})$.

\subsubsection{An adjunction with infinitesimal post-Lie algebras}

The next result shows that the space of primitive elements of an infinitesimal post-Hopf algebra is an infinitesimal post-Lie algebra, as expected.

\begin{proposition}\label{prop:primitive}
Let $(H,\wp,\bp)$ be an infinitesimal post-Hopf algebra. Then \[(P(H),\wt:=\wp|_{P(H)\ot P(H)},\bt:=\bp|_{P(H)\ot P(H)})\] is an infinitesimal post-Lie algebra. Moreover, the  assignment $P:\mathrm{IPHopf}\to\mathrm{IPLie}$ is functorial. 
\end{proposition}

\begin{proof}
    It is known that, since $(H,\wp)$ is a post-Hopf algebra, $(P(H),\wt:=\wp)$ is a post-Lie algebra, where $P(H)$ is a Lie algebra with bracket $[a,b]:=ab-ba$. We observe that, given $a,b\in P(H)$, we have
\[
\begin{split}
\Delta(a\bp b)&\overset{\eqref{deltalinear}}{=}(a_{1}\wp b_{1})\ot(a_{2}\bp b_{2})+(a_{1}\bp b_{1})\ot(a_{2}\wp b_{2})\\&=(a\wp b)\ot(1\bp1)+(a\wp1)\ot(1\bp b)+(1\wp1)\ot(a\bp b)+(1\wp b)\ot(a\bp1)\\&+(a\bp b)\ot(1\wp1)+(a\bp1)\ot(1\wp b)+(1\bp1)\ot(a\wp b)+(1\bp b)\ot(a\wp1)\\&=1\ot(a\bp b)+(a\bp b)\ot1
\end{split}
\]
since $a\bp1=1\bp a=0$ for all $a\in H$. Thus, we have a morphism $\bp:P(H)\ot P(H)\to P(H)$. We compute
\[
\begin{split}
a\bp[b,c]&=a\bp(bc)-a\bp(cb)\\&\overset{\eqref{multiplicativity}}{=}(a_{1}\bp b)(a_{2}\wp c)+(a_{1}\wp b)(a_{2}\bp c)-(a_{1}\bp c)(a_{2}\wp b)-(a_{1}\wp c)(a_{2}\bp b)\\&=(a\bp b)c+(1\bp b)(a\wp c)+(a\wp b)(1\bp c)+b(a\bp c)\\&-(a\bp c)b-(1\bp c)(a\wp b)-(a\wp c)(1\bp b)-c(a\bp b)\\&=(a\bp b)c+b(a\bp c)-(a\bp c)b-c(a\bp b) \\
&= [a \bp b,c] + [b, a \bp c],
\end{split}
\]
so \eqref{eq:infpostlie:1} is satisfied. Moreover, we obtain
\[
\begin{split}
&\big(a\bp(b\wp c)\big)+\big(a\wp(b\bp c)\big)-\big(b\bp(a\wp c)\big)-\big(b\wp(a\bp c)\big)\overset{\eqref{deformedassociativity}}{=}\\&(a_{1}(a_{2}\bp b))\wp c
+(a_{1}(a_{2}\wp b))\bp c -(b_{1}(b_{2}\bp a))\wp c-(b_{1}(b_{2}\wp a))\bp c= \\&(a\bp b)\wp c
+(ab)\bp c +(a\wp b)\bp c-(b\bp a)\wp c-(ba)\bp c -(b\wp a)\bp c=\\&([a,b]+(a\wp b)-(b\wp a))\bp c+((a\bp b)-(b\bp a))\wp c,
\end{split}
\]
hence also \eqref{eq:infpostlie:2} holds. Therefore, $(P(H),\wt:=\wp,\bt:=\bp)$ is an infinitesimal post-Lie algebra. Clearly, a functor $P:\mathrm{IPHopf}\to\mathrm{IPLie}$ is obtained that sends a morphism $f$ to its restriction $f|_{P(H)}$.
\end{proof}

Next, let $(\mathfrak{g},[\cdot,\cdot],\wt)$ be a post-Lie algebra. According to Theorem \ref{thm:postHopfonuniversalenveloping}, the universal enveloping algebra $\mathrm{U}(\mathfrak{g})$ is a post-Hopf algebra with $\wp\colon \mathrm{U}(\mathfrak{g})\otimes \mathrm{U}(\mathfrak{g})\to \mathrm{U}(\mathfrak{g})$ determined by $\wp|_{\mathfrak{g}\otimes\mathfrak{g}}=\wt$ and inductively extended by
\[
(xy)\wp z=x\wt (y\wt z)-(x\wt y)\wt z,
\]
as well as 
\[
x\wp(yz)=(x_1\wp y)(x_2\wp z)=(x\wt y)z+y(x\wt z)
\]
for all $x,y,z\in\mathfrak{g}$. The natural question arises whether $\mathrm{U}(\mathfrak{g})$ even becomes an infinitesimal post-Hopf algebra if $\mathfrak{g}$ is an infinitesimal post-Lie algebra. The affirmative answer to this question is given via the following result.

\begin{proposition}\label{thm:universal}
Let $(\mathfrak{g},[\cdot,\cdot],\wt,\bt)$ be an infinitesimal post--Lie algebra. Then $(\mathrm{U}(\mathfrak{g}),\wp,\bp)$ is an infinitesimal post-Hopf algebra, where $(\mathrm{U}(\mathfrak{g}),\wp)$ is the post-Hopf algebra defined as before, and $\bp\colon \mathrm{U}(\mathfrak{g})\otimes \mathrm{U}(\mathfrak{g})\to \mathrm{U}(\mathfrak{g})$ is determined by $\bp|_{\mathfrak{g}\otimes\mathfrak{g}}=\bt$ and extended inductively as $1\bp x=0$,
\begin{equation}\label{rhdext1}
    (xy)\bp z
    =x\bt(y\wt z)-(x\bt y)\wt z
    +x\wt(y\bt z)-(x\wt y)\bt z
\end{equation}
and $x \bp1=0$,
\begin{equation}\label{rhdext2}
    x \bp(yz)
    =(x_1 \bp y)(x_2\wp z)
    +(x_1\wp y)(x_2 \bp z)
    =(x\bt y)z+y(x\bt z)
\end{equation}
for all $x,y,z\in\mathfrak{g}$. Moreover, the assignment $\mathrm{U}:\mathrm{IPLie}\to\mathrm{IPHopf}$ is functorial.
\end{proposition}

\begin{proof}
By Theorem \ref{thm:postHopfonuniversalenveloping}, there is a post-Hopf algebra $(\mathrm{U}(\mathfrak{g}),\wp)$, as described before. It remains to show that \eqref{rhdext1} and \eqref{rhdext2} describe an infinitesimal structure $ \bp\ \colon \mathrm{U}(\mathfrak{g})\otimes \mathrm{U}(\mathfrak{g})\to \mathrm{U}(\mathfrak{g})$ for this post-Hopf algebra. We first prove that \eqref{rhdext1} and \eqref{rhdext2} are well-defined. Explicitly, \eqref{rhdext1} and \eqref{rhdext2} are well-defined on the tensor algebra, and we show that they descend to the quotient $\mathrm{U}(\mathfrak{g})$. Starting with \eqref{rhdext2}, we prove that $yz-zy-[y,z]$, with $y,z\in\mathfrak{g}$, is in the kernel of $x \bp(\cdot)\colon T(\mathfrak{g})\to \mathrm{U}(\mathfrak{g})$:
\begin{align*}
    x \bp(yz-zy-[y,z])
    &=(x\bt y)z+y(x\bt z)
    -(x\bt z)y-z(x\bt y)
    -x\bt[y,z]\\
    &\overset{\eqref{eq:infpostlie:1}}{=}(x\bt y)z+y(x\bt z)
    -(x\bt z)y-z(x\bt y)
    -[x\bt y,z]-[y,x\bt z]\\
    &=0,
\end{align*}
using that $[x,y]=xy-yx$ on the universal enveloping algebra. More precisely, let $I$ be the defining ideal of $\mathrm{U}(\mathfrak g)$
inside $T(\mathfrak g)$. For each $x\in\mathfrak g$, the map
$y\mapsto x\bt y$ extends uniquely to a derivation of
$T(\mathfrak g)$. The preceding computation shows that this
derivation sends each defining relation into $I$. By the Leibniz
rule, it therefore preserves the whole ideal $I$, and hence
induces a derivation of $\mathrm{U}(\mathfrak g)$.
Thus,
\begin{equation}\label{Ugext2}
    x \bp(\cdot)\colon \mathrm{U}(\mathfrak{g})\to \mathrm{U}(\mathfrak{g}),\qquad
    x^1\cdots x^n\mapsto x \bp(x^1\cdots x^n):=\sum_{i=1}^nx^1\cdots(x\bt x^i)\cdots x^n
\end{equation}
is well-defined for all $x\in\mathfrak{g}$. 

Similarly for \eqref{rhdext1}, let $x,y,z\in\mathfrak{g}$. We have to show that $xy-yx-[x,y]$ is in the kernel of $(\cdot) \bp z\colon T(\mathfrak{g})\to\mathfrak{g}$. In fact,
\begin{align*}
    (xy-yx-[x,y]) \bp z
    &=x\bt(y \wt z)-(x\bt y) \wt z
    +x \wt (y\bt z)-(x \wt y)\bt z\\
    &\quad-(y\bt(x \wt z)-(y\bt x) \wt z
    +y \wt (x\bt z)-(y \wt x)\bt z)\\
    &\quad-[x,y]\bt z\\
    &\overset{\eqref{eq:infpostlie:2}}{=}0.
\end{align*}
Inductively, we extend \eqref{rhdext1} to
\begin{equation}\label{Ugext1}
\begin{split}
    (x^1\cdots x^n) \bp u&:=x^1 \bp((x^2\cdots x^n)\wp u)
+x^1\wp ((x^2\cdots x^n) \bp u)\\
&\quad-(x^1 \bp(x^2\cdots x^n))\wp u
-(x^1\wp(x^2\cdots x^n)) \bp u,
\end{split}
\end{equation}
with $n>0$, $x^1,\ldots,x^n\in\mathfrak{g}$ and $u\in \mathrm{U}(\mathfrak{g})$. 
The enveloping construction of
Theorem~\ref{thm:postHopfonuniversalenveloping} also applies
over $R=\Bbbk[\hbar]/(\hbar^2)$ to the free $R$-module
$R\otimes\mathfrak g$, since the same PBW argument
and post-Lie identities remain valid over $R$.
Applying this construction to
$x\wt_\hbar y=x\wt y+\hbar(x\bt y)$ and taking the coefficient of $\hbar$,
under the canonical identification
$\mathrm{U}_R(R\otimes\mathfrak g)
\cong R\otimes \mathrm{U}(\mathfrak g)$,
gives a well-defined operation $\bp$ satisfying
\eqref{Ugext1}, \eqref{Ugext2} and the infinitesimal
post-Hopf identities \eqref{deltalinear},
\eqref{multiplicativity}, \eqref{deformedassociativity}.
For completeness, we verify these identities explicitly
on primitive elements $x,y,z\in\mathfrak g$. Using $1 \bp x=0=x \bp 1=1 \bp 1$, we obtain
\begin{align*}
    (x_1\wp  y_1)\otimes(x_2 \bp y_2)+(x_1 \bp y_1)\otimes(x_2\wp y_2)
    =1\otimes(x\bt y)+(x\bt y)\otimes 1
    =\Delta(x\bt y)=\Delta(x\bp y)
\end{align*}
since $x\bt y\in\mathfrak{g}$ is primitive,
\begin{align*}
    (x_1(x_2 \bp y))\wp z
    +(x_1(x_2\wp y)) \bp z
    &=(x\bt y) \wt z
    +(xy) \bp z
    +(x \wt y)\bt z\\
    &\overset{\eqref{rhdext1}}{=}x\bt(y \wt z)
    +x \wt(y\bt z)\\&=x\bp(y \wp z)
    +x \wp(y\bp z)
\end{align*}
and
\begin{align*}
    (x_1 \bp y)(x_2\wp z)
    +(x_1\wp y)(x_2 \bp z)
    =(x\bt y)z+y(x\bt z)
    \overset{\eqref{rhdext2}}{=}x\bp(yz).
\end{align*}
Let $f:\mathfrak g\to\mathfrak \g'$ be a morphism of
infinitesimal post-Lie algebras and put $\Phi=\mathrm{U}(f)$.
The map $\Phi$ already preserves $\wp$.
Formula \eqref{Ugext2} and preservation of $\bt$ give
\[
\Phi(x\bp b)=f(x)\bp\Phi(b)
\qquad(x\in\mathfrak g,\ b\in \mathrm{U}(\mathfrak g)).
\]
Induction on the PBW degree of the first variable,
uniformly in the second variable, now gives
\[
\Phi(a\bp b)=\Phi(a)\bp\Phi(b)
\qquad(a,b\in \mathrm{U}(\mathfrak g)).
\]
Indeed, the unit case is trivial,
and the induction step follows from \eqref{Ugext1},
preservation of $\wp$ and
$\deg(x\wp a)\leq\deg(a)$.
This proves functoriality and concludes the proof.
\end{proof}

Finally, we show that the functors $P:\mathrm{IPHopf}\to\mathrm{IPLie}$ and $\mathrm{U}:\mathrm{IPLie}\to\mathrm{IPHopf}$ constructed in Proposition~\ref{prop:primitive} and Proposition \ref{thm:universal}, respectively, provide an adjunction.

\begin{theorem}\label{thm:main}
The pair $(\mathrm{U},P)$ provides an adjunction between the categories $\mathrm{IPHopf}$ and $\mathrm{IPLie}$. More precisely, given an object $(\mathfrak{g},\wt,\bt)$ in $\mathrm{IPLie}$ and an object $(H,\wp,\bp)$ in $\mathrm{IPHopf}$, there is a natural bijection 
\[
\mathrm{Hom}_{\mathrm{IPHopf}}\big(\mathrm{U}(\mathfrak{g}),H\big)\cong\mathrm{Hom}_{\mathrm{IPLie}}\big(\mathfrak{g},P(H)\big).
\]
Moreover, if the base field has characteristic 0, the adjunction becomes an equivalence between cocommutative connected infinitesimal post-Hopf algebras and infinitesimal post-Lie algebras.
\end{theorem}

\begin{proof}
This amounts to extending the isomorphisms \eqref{def:Phi}-\eqref{def:Psi} to the infinitesimal setting.
Namely, we have to verify that, given $f \in \mathrm{Hom}_{\mathrm{IPHopf}}\big(\mathrm{U}(\mathfrak{g}),H\big)$, 
the morphism $\Phi(f)$ 
belongs to $\mathrm{Hom}_{\mathrm{IPLie}}\big(\mathfrak{g},P(H)\big)$ and, conversely, for any $f \in \mathrm{Hom}_{\mathrm{IPLie}}\big(\mathfrak{g},P(H)\big)$ 
the morphism $\Psi(f)$ belongs to $\mathrm{Hom}_{\mathrm{IPHopf}}\big(\mathrm{U}(\mathfrak{g}),H\big)$. The first assertion is trivial by construction. To show the second, we compute:
\begin{align*}
&\Psi(f)((xy)\bp z)
    \\&\overset{\eqref{rhdext1}}{=}f\big(x\bt(y\wt z)-(x\bt y)\wt z
    +x\wt(y\bt z)-(x\wt y)\bt z\big)\\&\hspace{0.15cm}=f(x)\bp(f(y)\wp f(z))-(f(x)\bp f(y))\wp f(z)
    +f(x)\wp(f(y)\bp f(z))-(f(x)\wp f(y))\bp f(z)\\&\overset{\eqref{deformedassociativity}}{=}(f(x)f(y))\bp f(z)=\Psi(f)(xy)\bp\Psi(f)(z)
\end{align*}
and
\begin{align*}
    \Psi(f)(x \bp(yz))
    &\overset{\eqref{rhdext2}}{=}\Psi(f)\big((x\bt y)z+y(x\bt z)\big)=f(x\bt y)f(z)+f(y)f(x\bt z)\\&=(f(x)\bp f(y))f(z)+f(y)(f(x)\bp f(z))\overset{\eqref{multiplicativity}}{=}f(x)\bp(f(y)f(z))\\&=\Psi(f)(x)\bp\Psi(f)(yz).
\end{align*}

These identities extend to all $u,v\in \mathrm{U}(\g)$ as follows. For $x\in\g$, formula \eqref{Ugext2} and multiplicativity of $\Psi(f)$ show that
\[
\Psi(f)(x\bp v)=f(x)\bp\Psi(f)(v)
\qquad\text{for every }v\in \mathrm{U}(\g).
\]
Here one uses that $f(x)$ is primitive, so \eqref{multiplicativity} makes $f(x)\bp(-)$ a derivation. We then use induction on the PBW degree of $u$, uniformly for all $v$. The case $u=1$ is trivial, and the case $u=x\in\g$ has just been proved. For $u=xa$, formula \eqref{Ugext1} expresses $u\bp v$ using the first variables $x$, $a$, $x\bp a$ and $x\wp a$. The last three have degree at most that of $a$. The inductive hypothesis, preservation of $\wp$ by $\Psi(f)$, and the same identity in $H$ with the primitive element $f(x)$ therefore prove
\[
\Psi(f)(u\bp v)=\Psi(f)(u)\bp\Psi(f)(v).
\]
Thus $\Psi(f)$ is indeed a morphism of infinitesimal post-Hopf algebras.

If we assume that $\Bbbk$ has characteristic 0, we know that the $H$-component of the counit $\epsilon_{H}$ is an isomorphism in $\mathrm{PHopf}$, for any cocommutative connected infinitesimal post-Hopf algebra $(H,\wp,\bp)$, and the $\g$-component of the unit $\eta_{\g}$ is the identity as a morphism in $\mathrm{PLie}$, for any infinitesimal post-Lie algebra $(\g,\wt,\bt)$ (see Remark \ref{rmk:MilnorMoorepostHopf}). Since $\epsilon_{H}$ is a morphism in $\mathrm{IPHopf}$ and $\eta_{\g}$ is a morphism in $\mathrm{IPLie}$, we can conclude the proof.
\end{proof}

\section{On the operad $\mathcal{IPL}$ of infinitesimal post-Lie algebras}\label{sec:operad}
In this section, we first give a brief introduction to linear symmetric operads. We then define the quadratic operad $\mathcal{IPL}$ of infinitesimal post-Lie algebras, and we show its Koszulity using filtered distributive laws \cite{DG}. A precise treatment of operads is out of the scope of this paper; for more details, we refer to \cite{MSS,LodayVallette,Felicia}.
\subsection{Linear symmetric operads}
Let us start with a precise definition of linear symmetric operads. 
\begin{definition}[{\cite[Definition 1.4]{MSS}}]
Let $\Bbbk$ be a field. A \textbf{linear symmetric operad} $\mathcal P$
(over $\Bbbk$) consists of:

\begin{enumerate}
\item A family of vector spaces $\mathcal{P}=\{ \mathcal{P}(n)\}_{n \geq 0}$.
\item For each $n \geq 0$, a right action of the symmetric group $\mathbb{S}_n$ on
$ \mathcal{P}(n)$:
\[
 \mathcal{P}(n) \times \mathbb{S}_n \to \mathcal{P}(n),
\qquad (f,\sigma)\mapsto f\cdot\sigma.
\]

\item A family of linear maps $\{ \gamma_{k_1,\ldots,k_n}\}_{n,k_1,\ldots,k_n \geq 0}$, (called composition maps)
\begin{align*}
\gamma_{k_1,\ldots,k_n}: \mathcal{P}(n) \ten \mathcal{P}(k_1) \ten \cdots \ten \mathcal{P}(k_n) &\to \mathcal{P}(k_1 + \cdots + k_n).
\end{align*}

\item A unit element $\mathbf 1\in \mathcal{P}(1)$.
\end{enumerate}

We shall use the notation $\gamma(f;g_1,\ldots,g_n):= \gamma_{k_1,\ldots,k_n}(f \ten g_1 \ten \ldots \ten g_n)$. These data are required to satisfy the following properties:

\paragraph{Associativity.}
For every $f\in\mathcal P(n)$, $g_i\in\mathcal P(k_i)$, and $h_{ij}\in\mathcal P(l_{ij})$ (where $1 \leq i \leq n$ and $1 \leq j \leq k_i$), one has
\[
\gamma\!\bigl(
\gamma(f;g_1,\dots,g_n);
h_{11},\dots,h_{n k_n}
\bigr)
=
\gamma\!\bigl(
f;
\gamma(g_1;h_{11},\dots,h_{1k_1}),
\dots,
\gamma(g_n;h_{n1},\dots,h_{nk_n})
\bigr).
\]

\paragraph{Unit.}
For every $f\in\mathcal P(n)$, one has 
\[
\gamma(\mathbf 1;f)=\gamma(f;\mathbf 1,\dots,\mathbf 1)=f.
\]

\paragraph{Equivariance.}
For every $\sigma \in \mathbb{S}_n$ and
$\sigma_i \in \mathbb{S}_{k_i}$ (with $1 \le i \le n$), one has
\begin{align*}
\gamma(f \cdot \sigma; g_1,\ldots,g_n)
&=
\gamma\bigl(
f;
g_{\sigma^{-1}(1)},
\ldots,
g_{\sigma^{-1}(n)}
\bigr)\cdot \sigma\langle k_{1},\ldots, k_{n}\rangle,\\
\gamma(f;
g_1 \cdot \sigma_1,
\ldots,
g_n \cdot \sigma_n
)
&=
\gamma(f;g_1,\ldots,g_n)
\cdot
(\sigma_1 \oplus \cdots \oplus \sigma_n),
\end{align*}
where $\sigma\langle k_{1},\ldots ,k_{n}\rangle\in \mathbb{S}_{k_{1}+\cdots+k_{n}}$ moves the $i$-th block to position $\sigma(i)$, preserving the order within each block, and $\sigma_1 \oplus \cdots \oplus \sigma_n
\in
\mathbb{S}_{k_1+\cdots+k_n}$
denotes the block permutation acting by $\sigma_i$ on the
$i$-th block of size $k_i$.
\end{definition}
If $\mu \in \mathcal{P}(n)$, we say that $\mu$ has \textbf{arity} $n$. 
We now briefly recall the description of symmetric operads by generators
and relations.
Let $\mathcal{V}=(\mathcal{V}(n))_{n\geq 0}$ be an $\mathbb S$-module, i.e.\ a family of vector spaces endowed with right actions of the symmetric groups $\mathbb{S}_n$.
Following \cite[1.9]{MSS}, the free symmetric operad generated by $\mathcal{V}$, denoted by $\mathcal{F}(\mathcal{V})$, is the operad whose elements are finite linear combinations of decorated rooted trees whose vertices are labeled by elements of $\mathcal{V}$, where the composition maps are given by grafting trees, and the actions of the symmetric groups are induced by permuting the labels of the leaves. The operad $\mathcal{F}(\mathcal{V})$ has the following universal property (see \cite[5.5]{LodayVallette}): for every symmetric operad $\mathcal P$ and every morphism of $\mathbb S$-modules $f:\mathcal{V}\to \mathcal{P}$
there exists a unique operad morphism $\widetilde f:\mathcal{F}(\mathcal{V})\to \mathcal{P}$
such that $\widetilde f\circ\iota=f$,
where $\iota:\mathcal{V} \hookrightarrow\mathcal{F}(\mathcal{V})$ denotes the canonical
inclusion. Following \cite[5.2.14]{LodayVallette}, an operadic ideal of an operad $\mathcal P$ is a family
$I=(I(n))_{n\ge0}$ of $\mathbb S_n$-stable subspaces such that $\gamma(f;g_1,\ldots,g_n)\in I$
whenever at least one of the elements
$f,g_1,\ldots,g_n$ belongs to $I$.
\begin{definition}
Let $\mathcal{V}$ be an $\mathbb S$-module and let $\mathcal{R}\subseteq \mathcal{F}(\mathcal{V})$
be a subspace. The \emph{operadic ideal generated by $\mathcal{R}$}, denoted by $(\mathcal{R})$, is the smallest operadic ideal of $\mathcal{F}(\mathcal{V})$ containing $\mathcal{R}$.
The quotient operad $\mathcal{P}=\mathcal{F}(\mathcal{V})/(\mathcal{R})$
is called the symmetric operad generated by $\mathcal{V}$ subject to the relations $\mathcal{R}$ (see \cite[5.1]{Felicia}). If $\mathcal{V}$ is concentrated in arity $2$, i.e.\ $\mathcal{V}(n) = 0$ iff $n \neq 2$, and  $\mathcal{R} \subseteq \mathcal{F}(\mathcal{V})(3)$, we say that $\mathcal{P}$ is a \textbf{binary quadratic operad} (see \cite[Definition 3.31]{MSS}).
\end{definition}
Using the universal property of the free operad, one can prove that
any symmetric operad $\mathcal{P}$ admits a presentation $\mathcal{P} \cong \mathcal{F}(\mathcal{V})/(\mathcal{R})$ (see e.g. \cite[Theorem 5.20]{Felicia}). Next, for a binary quadratic operad $\mathcal{P}=\mathcal{F}(\mathcal{V})/(\mathcal{R})$ with $\mathrm{dim}_{\Bbbk}\mathcal{V}(2)<\infty$,
one can construct its \emph{quadratic dual operad} $\mathcal P^!
=
\mathcal F(\mathcal{V}^\vee)/(\mathcal{R}^\perp)$,
where $\mathcal{V}^{\vee}$ is concentrated in arity 2, with $\mathcal{V}^{\vee}(2) =\mathcal{V}(2)^{*}\ot\mathrm{sgn}_{2}$ and
$\mathcal{R}^\perp$ is the annihilator of $\mathcal{R}$ under the corresponding sign-twisted pairing (see \cite[Definition 3.37]{MSS}).

\begin{definition}[{\cite[Definition 3.40]{MSS} and \cite[Theorem 7.4.6]{LodayVallette}}]
Let $\mathcal{P}$ be a quadratic operad and let
$\mathcal{P}^!$ be its quadratic dual operad.
Denote by $\mathcal{P}^{\text{\textexclamdown}}$ the
corresponding Koszul dual cooperad.
The operad $\mathcal{P}$ is said to be
\textbf{Koszul} if the canonical morphism $\Omega(\mathcal{P}^{{\text{\textexclamdown}}})
\longrightarrow \mathcal{P}$
is a quasi-isomorphism.
\end{definition}
Informally, a Koszul operad is a quadratic operad whose algebraic
structure is completely governed by its quadratic relations, with no
hidden higher-order homological relations.
\begin{remark}
When a quadratic operad attached to some algebraic structure is Koszul, there exists a cohomology complex, the so-called Andr\'e-Quillen cohomology \cite{Milles}, controlling the first-order deformation theory of such structure. B. Vallette \cite{Vallette} showed that the algebraic operad $\mathcal{PL}$ of post-Lie algebras is Koszul (and, similarly, F. Chapoton and M. Livernet proved the analogue result for the operad of pre-Lie algebras, see \cite{CL}). A realization of this cohomology for post-Lie algebras can be found in \cite{LazST}.
\end{remark}
We now recall filtered distributive laws \cite{DG}. 
Assume $\Bbbk$ is a field of characteristic zero. Let $\circ$ denote the composition of $\mathbb{S}$-modules (see \cite[5.1.4]{LodayVallette}), $\circ_i$ denote the partial operadic composition (see \cite[5.3.4]{LodayVallette}), and, for two subcollections $\mathcal{U}_1, \mathcal{U}_2$ of an operad $\mathcal{P}$, let  $\mathcal{U}_1 \bullet \mathcal{U}_2$ denote the subcollection of $\mathcal{P}$ spanned by all elements $\varphi \circ_i \psi$, with $\varphi \in \mathcal{U}_1$, $\psi \in \mathcal{U}_2$.

\begin{definition}
\label{definition-filtered-law}
Let $\mathcal{A}= \mathcal{F}(\mathcal{V})/(\mathcal{R})$ and $\mathcal{B}= \mathcal{F}(\mathcal{W})/ (\mathcal{S})$ be two quadratic operads. Consider a pair of $\mathbb{S}$-modules maps
\begin{align}
&s: \mathcal{R} \to \mathcal{W} \bullet \mathcal{V} \oplus \mathcal{V} \bullet \mathcal{W} \oplus \mathcal{W} \bullet \mathcal{W} \label{eq:filtered-law-one}\\
&d:\mathcal{W} \bullet \mathcal{V} \to \mathcal{V} \bullet \mathcal{W} \oplus \mathcal{W} \bullet \mathcal{W}. \label{eq:filtered-law-two}
\end{align}
Set $\mathcal{G}=\{x - s(x), x \in \mathcal{R} \}$ and $\mathcal{D}=\{x - d(x), x \in \mathcal{W} \bullet \mathcal{V} \}$ and consider the quadratic operad $\mathcal{E} = \mathcal{F}(\mathcal{V} \oplus \mathcal{W})/(\mathcal{G} \oplus \mathcal{D} \oplus \mathcal{S})$. 
The projection killing $\mathcal{W}$ induces a surjection $\pi:\mathcal{E}\to\mathcal{A}$. Choose an $\mathbb{S}$-module section $i:\mathcal{A}\to\mathcal{E}$ of $\pi$ (which always exists in characteristic zero). Composing the representatives $i(a)$ with the images of the $\mathcal{B}$-operations under the natural operad morphism $\mathcal{B}\to\mathcal{E}$ yields the surjection of $\mathbb{S}$-modules $\xi:\mathcal{A}\circ\mathcal{B}\to\mathcal{E}$ (when $s= 0$, one may take the canonical section induced by the inclusion of $\mathcal{V}$). The pair $(s,d)$ is called a \textbf{filtered distributive law} between $\mathcal{A}$ and $\mathcal{B}$ if the restriction of $\xi$ to elements of arity $4$ (i.e.\ weight-$3$ elements) $\xi_3: 
(\mathcal{A}\circ\mathcal{B})(4) \twoheadrightarrow \mathcal{E}(4)$ is an isomorphism of $\mathbb{S}_{4}$-modules.
\end{definition}

The following result allows one to prove the Koszulity of the operad $\mathcal{E}$ through that of $\mathcal{A}$ and $\mathcal{B}$:
\begin{theorem}[{\cite[Theorem 5.2]{DG}}]
\label{theorem-Koszulity}
Let $\mathcal{A}= \mathcal{F}(\mathcal{V})/(\mathcal{R})$ and $\mathcal{B}= \mathcal{F}(\mathcal{W})/ (\mathcal{S})$ be two Koszul operads and $(s,d)$ be a filtered distributive law between $\mathcal{A}$ and $\mathcal{B}$. Then the quadratic operad $\mathcal{E}$ built as above is Koszul, and $\mathcal{A} \circ \mathcal{B} \cong \mathcal{E}$ as $\mathbb{S}$-modules.
\end{theorem}

Another approach to proving the Koszulity of a quadratic operad relies on the Hoffbeck-Dotsenko-Khoroshkin criterion. In a nutshell, building on the work of V. Dotsenko and A. Khoroshkin \cite{DK}, one can apply an operadic analogue of Buchberger's algorithm \cite{Buch} to compute Gr\"obner bases for operads. Hoffbeck's PBW criterion \cite{Hoff} then implies that any quadratic operad admitting a quadratic Gr\"obner basis is Koszul. Such a technique has also been implemented in Haskell by V. Dotsenko and M. Vejdemo-Johansson \cite{DVJ}.

\subsection{The operad $\mathcal{IPL}$ and its Koszulity}
Let $\mathcal{V}$ be the $\mathbb S$-module concentrated in arity $2$ and defined by $\mathcal{V}(2)= \Bbbk \langle [\cdot,\cdot],\wt,\wt^{\mathrm{op}},\bt, \bt^{\mathrm{op}}\rangle$ and $\mathcal{V}(n)=0$ for $n\neq 2$, where the action of the transposition $(12)\in\mathbb S_2$ on $\mathcal{V}(2)$ is given by
\[
[\cdot,\cdot] \cdot (12)=-[\cdot,\cdot],
\quad
\wt \cdot (12)=\wt^{\mathrm{op}},
\quad
\bt \cdot (12)=\bt^{\mathrm{op}}.
\]
In other words, as an $\mathbb{S}_2$-module we have $\mathcal{V}(2) \cong \mathrm{sgn} \oplus \Bbbk[\mathbb{S}_2] \oplus \Bbbk[\mathbb{S}_2]$.
\begin{definition}
The operad $\mathcal{IPL}$ is defined as the quotient $\mathcal{IPL} = \mathcal{F}(\mathcal{V})/(\mathcal{R})$, where $\mathcal{V}$ is the $\mathbb{S}$-module as above, and $\mathcal{R}$ is the $\mathbb{S}_{3}$-submodule of $\mathcal{F}(\mathcal{V})(3)$ generated by the Jacobi identity and the relations \eqref{eq:postlie:1}, \eqref{eq:postlie:2}, \eqref{eq:infpostlie:1}, \eqref{eq:infpostlie:2}.
\end{definition}
From now on, we will simply regard $\mathcal{V}$ as being generated by $[\cdot,\cdot], \wt, \bt$, with the opposite operations $\wt^\mathrm{op}, \bt^{\mathrm{op}}$ understood as the images of $\wt,\bt$ under the action of the transposition $(12)$. By construction $\mathcal{IPL}$ is a quadratic operad. We describe $\mathcal{IPL}$ using trees and grafting trees operations as follows: the generators $[\cdot,\cdot],\wt,\bt$ are represented by the following $2$-corollas
\[
\small
\begin{tikzpicture}[scale=1]

\node[circle,draw,inner sep=2pt] (b) at (0,0) {${\scriptstyle[\cdot,\cdot]}$};
\draw (-0.6,0.8)--(b);
\draw (0.6,0.8)--(b);
\draw (b)--(0,-1);

\node[circle,draw,inner sep=2pt] (w) at (4,0) {$\wt$};
\draw (3.4,0.8)--(w);
\draw (4.6,0.8)--(w);
\draw (w)--(4,-1);

\node[circle,draw,inner sep=2pt] (k) at (8,0) {$\bt$};
\draw (7.4,0.8)--(k);
\draw (8.6,0.8)--(k);
\draw (k)--(8,-1);

\end{tikzpicture}
\]
whence the action of the transposition on the Lie bracket is represented by 
\[
\small
\begin{tikzpicture}[scale=1]

\node[circle,draw,inner sep=2pt] (b) at (0,0) {${\scriptstyle[\cdot,\cdot]}$};
\draw (-0.6,0.8)--(b);
\draw (0.6,0.8)--(b);
\draw (b)--(0,-1);
\node at (-0.6,1) {$2$};
\node at (0.6,1) {$1$};
\node at (1,0) {$=$};
\node at (1.8,0) {$-$};
\begin{scope}[shift={(2.5,0)}]
\node[circle,draw,inner sep=2pt] (b) at (0,0) {${\scriptstyle[\cdot,\cdot]}$};
\draw (-0.6,0.8)--(b);
\draw (0.6,0.8)--(b);
\draw (b)--(0,-1);
\node at (-0.6,1) {$1$};
\node at (0.6,1) {$2$};
\end{scope}

\end{tikzpicture}
\]
The Jacobi relation is thus represented by 
\[
\small
\begin{tikzpicture}[scale=1]

\node[circle,draw,inner sep=2pt] (b1) at (0,0) {$\scriptstyle[\cdot,\cdot]$};
\draw (b1)--(0,-0.8);

\node[circle,draw,inner sep=2pt] (a1) at (-0.7,0.9) {$\scriptstyle[\cdot,\cdot]$};
\draw (a1)--(b1);
\draw (-1.3,1.7)--(a1);
\draw (-0.1,1.7)--(a1);

\draw (0.7,0.9)--(b1);

\node at (-1.3,1.9) {$1$};
\node at (-0.1,1.9) {$2$};
\node at (0.7,1.1) {$3$};

\node at (1.6,0.5) {$+$};

\begin{scope}[shift={(3.2,0)}]

\node[circle,draw,inner sep=2pt] (b2) at (0,0) {$\scriptstyle[\cdot,\cdot]$};
\draw (b2)--(0,-0.8);

\node[circle,draw,inner sep=2pt] (a2) at (-0.7,0.9) {$\scriptstyle[\cdot,\cdot]$};
\draw (a2)--(b2);
\draw (-1.3,1.7)--(a2);
\draw (-0.1,1.7)--(a2);

\draw (0.7,0.9)--(b2);

\node at (-1.3,1.9) {$2$};
\node at (-0.1,1.9) {$3$};
\node at (0.7,1.1) {$1$};

\end{scope}

\node at (4.8,0.5) {$+$};

\begin{scope}[shift={(6.4,0)}]

\node[circle,draw,inner sep=2pt] (b3) at (0,0) {$\scriptstyle[\cdot,\cdot]$};
\draw (b3)--(0,-0.8);

\node[circle,draw,inner sep=2pt] (a3) at (-0.7,0.9) {$\scriptstyle[\cdot,\cdot]$};
\draw (a3)--(b3);
\draw (-1.3,1.7)--(a3);
\draw (-0.1,1.7)--(a3);

\draw (0.7,0.9)--(b3);

\node at (-1.3,1.9) {$3$};
\node at (-0.1,1.9) {$1$};
\node at (0.7,1.1) {$2$};

\end{scope}

\node at (8.1,0.5) {$=0$};

\end{tikzpicture}
\]
whence the post-Lie relations \eqref{eq:postlie:1} and \eqref{eq:postlie:2} are respectively described by
\[
\footnotesize
\begin{tikzpicture}[scale=1]

\node[circle,draw,inner sep=2pt] (w) at (0,0) {$\wt$};

\node[circle,draw,inner sep=2pt] (b) at (0.8,1) {$\scriptstyle[\cdot,\cdot]$};

\draw (w) -- (0,-1);
\draw (-0.8,1) -- (w);
\draw (b) -- (w);
\draw (0.2,2) -- (b);
\draw (1.4,2) -- (b);

\node at (-0.8,1.2) {$1$};
\node at (0.2,2.2) {$2$};
\node at (1.4,2.2) {$3$};

\node at (2.5,0.5) {$=$};

\begin{scope}[shift={(4,0)}]

\node[circle,draw,inner sep=2pt] (b1) at (0,0) {$\scriptstyle[\cdot,\cdot]$};

\node[circle,draw,inner sep=2pt] (w1) at (-0.8,1) {$\wt$};

\draw (b1) -- (0,-1);
\draw (w1) -- (b1);
\draw (0.8,1) -- (b1);

\draw (-1.4,2) -- (w1);
\draw (-0.2,2) -- (w1);

\node at (-1.4,2.2) {$1$};
\node at (-0.2,2.2) {$2$};
\node at (0.8,1.2) {$3$};

\end{scope}

\node at (6.2,0.5) {$+$};

\begin{scope}[shift={(8,0)}]

\node[circle,draw,inner sep=2pt] (b2) at (0,0) {$\scriptstyle[\cdot,\cdot]$};

\node[circle,draw,inner sep=2pt] (w2) at (0.8,1) {$\wt$};

\draw (b2) -- (0,-1);
\draw (-0.8,1) -- (b2);
\draw (w2) -- (b2);

\draw (0.2,2) -- (w2);
\draw (1.4,2) -- (w2);

\node at (-0.8,1.2) {$2$};
\node at (0.2,2.2) {$1$};
\node at (1.4,2.2) {$3$};

\end{scope}

\end{tikzpicture}
\]
and 
\[
\small
\begin{tikzpicture}[scale=0.8]

\node[circle,draw,inner sep=2pt] (w) at (0,0) {$\wt$};

\node[circle,draw,inner sep=2pt] (b) at (-0.8,1)
{$\scriptstyle[\cdot,\cdot]$};

\draw (w)--(0,-1);

\draw (b)--(w);
\draw (0.8,1)--(w);

\draw (-1.4,2)--(b);
\draw (-0.2,2)--(b);

\node at (-1.4,2.2) {$1$};
\node at (-0.2,2.2) {$2$};
\node at (0.8,1.2) {$3$};

\node at (2,0.5) {$=$};

\begin{scope}[shift={(3.5,0)}]

\node[circle,draw,inner sep=2pt] (w1) at (0,0) {$\wt$};

\node[circle,draw,inner sep=2pt] (w11) at (0.8,1) {$\wt$};

\draw (w1)--(0,-1);

\draw (-0.8,1)--(w1);
\draw (w11)--(w1);

\draw (0.2,2)--(w11);
\draw (1.4,2)--(w11);

\node at (-0.8,1.2) {$1$};
\node at (0.2,2.2) {$2$};
\node at (1.4,2.2) {$3$};

\end{scope}

\node at (5.5,0.5) {$-$};

\begin{scope}[shift={(7,0)}]

\node[circle,draw,inner sep=2pt] (w2) at (0,0) {$\wt$};

\node[circle,draw,inner sep=2pt] (w22) at (0.8,1) {$\wt$};

\draw (w2)--(0,-1);

\draw (-0.8,1)--(w2);
\draw (w22)--(w2);

\draw (0.2,2)--(w22);
\draw (1.4,2)--(w22);

\node at (-0.8,1.2) {$2$};
\node at (0.2,2.2) {$1$};
\node at (1.4,2.2) {$3$};

\end{scope}

\node at (9,0.5) {$-$};

\begin{scope}[shift={(11,0)}]

\node[circle,draw,inner sep=2pt] (w3) at (0,0) {$\wt$};

\node[circle,draw,inner sep=2pt] (w33) at (-0.8,1) {$\wt$};

\draw (w3)--(0,-1);

\draw (w33)--(w3);
\draw (0.8,1)--(w3);

\draw (-1.4,2)--(w33);
\draw (-0.2,2)--(w33);

\node at (-1.4,2.2) {$1$};
\node at (-0.2,2.2) {$2$};
\node at (0.8,1.2) {$3$};

\end{scope}

\node at (12.5,0.5) {$+$};

\begin{scope}[shift={(14.5,0)}]

\node[circle,draw,inner sep=2pt] (w4) at (0,0) {$\wt$};

\node[circle,draw,inner sep=2pt] (w44) at (-0.8,1) {$\wt$};

\draw (w4)--(0,-1);

\draw (w44)--(w4);
\draw (0.8,1)--(w4);

\draw (-1.4,2)--(w44);
\draw (-0.2,2)--(w44);

\node at (-1.4,2.2) {$2$};
\node at (-0.2,2.2) {$1$};
\node at (0.8,1.2) {$3$};

\end{scope}

\end{tikzpicture}
\]
Finally, relations \eqref{eq:infpostlie:1} and \eqref{eq:infpostlie:2} are respectively represented by 
\[
\footnotesize
\begin{tikzpicture}[scale=1]

\node[circle,draw,inner sep=2pt] (w) at (0,0) {$\bt$};

\node[circle,draw,inner sep=2pt] (b) at (0.8,1) {$\scriptstyle[\cdot,\cdot]$};

\draw (w) -- (0,-1);
\draw (-0.8,1) -- (w);
\draw (b) -- (w);
\draw (0.2,2) -- (b);
\draw (1.4,2) -- (b);

\node at (-0.8,1.2) {$1$};
\node at (0.2,2.2) {$2$};
\node at (1.4,2.2) {$3$};

\node at (2.5,0.5) {$=$};

\begin{scope}[shift={(4,0)}]

\node[circle,draw,inner sep=2pt] (b1) at (0,0) {$\scriptstyle[\cdot,\cdot]$};

\node[circle,draw,inner sep=2pt] (w1) at (-0.8,1) {$\bt$};

\draw (b1) -- (0,-1);
\draw (w1) -- (b1);
\draw (0.8,1) -- (b1);

\draw (-1.4,2) -- (w1);
\draw (-0.2,2) -- (w1);

\node at (-1.4,2.2) {$1$};
\node at (-0.2,2.2) {$2$};
\node at (0.8,1.2) {$3$};

\end{scope}

\node at (6.2,0.5) {$+$};

\begin{scope}[shift={(8,0)}]

\node[circle,draw,inner sep=2pt] (b2) at (0,0) {$\scriptstyle[\cdot,\cdot]$};

\node[circle,draw,inner sep=2pt] (w2) at (0.8,1) {$\bt$};

\draw (b2) -- (0,-1);
\draw (-0.8,1) -- (b2);
\draw (w2) -- (b2);

\draw (0.2,2) -- (w2);
\draw (1.4,2) -- (w2);

\node at (-0.8,1.2) {$2$};
\node at (0.2,2.2) {$1$};
\node at (1.4,2.2) {$3$};

\end{scope}

\end{tikzpicture}
\]
and
\[
\begin{tikzpicture}[scale=0.8]

\node[circle,draw,inner sep=2pt] (bt0) at (0,0) {$\bt$};
\node[circle,draw,inner sep=2pt] (br) at (-0.8,1)
{$\scriptstyle[\cdot,\cdot]$};

\draw (bt0)--(0,-1);
\draw (br)--(bt0);
\draw (0.8,1)--(bt0);

\draw (-1.4,2)--(br);
\draw (-0.2,2)--(br);

\node at (-1.4,2.2) {$1$};
\node at (-0.2,2.2) {$2$};
\node at (0.8,1.2) {$3$};

\node at (2,0.5) {$=$};

\begin{scope}[shift={(3.5,0)}]

\node[circle,draw,inner sep=2pt] (w1) at (0,0) {$\wt$};
\node[circle,draw,inner sep=2pt] (b1) at (0.8,1) {$\bt$};

\draw (w1)--(0,-1);
\draw (-0.8,1)--(w1);
\draw (b1)--(w1);
\draw (0.2,2)--(b1);
\draw (1.4,2)--(b1);

\node at (-0.8,1.2) {$1$};
\node at (0.2,2.2) {$2$};
\node at (1.4,2.2) {$3$};

\end{scope}

\node at (5.5,0.5) {$-$};

\begin{scope}[shift={(7,0)}]

\node[circle,draw,inner sep=2pt] (w2) at (0,0) {$\wt$};
\node[circle,draw,inner sep=2pt] (b2) at (0.8,1) {$\bt$};

\draw (w2)--(0,-1);
\draw (-0.8,1)--(w2);
\draw (b2)--(w2);
\draw (0.2,2)--(b2);
\draw (1.4,2)--(b2);

\node at (-0.8,1.2) {$2$};
\node at (0.2,2.2) {$1$};
\node at (1.4,2.2) {$3$};

\end{scope}

\node at (9,0.5) {$-$};

\begin{scope}[shift={(11,0)}]

\node[circle,draw,inner sep=2pt] (b3) at (0,0) {$\bt$};
\node[circle,draw,inner sep=2pt] (w3) at (-0.8,1) {$\wt$};

\draw (b3)--(0,-1);
\draw (w3)--(b3);
\draw (0.8,1)--(b3);
\draw (-1.4,2)--(w3);
\draw (-0.2,2)--(w3);

\node at (-1.4,2.2) {$1$};
\node at (-0.2,2.2) {$2$};
\node at (0.8,1.2) {$3$};

\end{scope}

\node at (12.5,0.5) {$+$};

\begin{scope}[shift={(14.5,0)}]

\node[circle,draw,inner sep=2pt] (b4) at (0,0) {$\bt$};
\node[circle,draw,inner sep=2pt] (w4) at (-0.8,1) {$\wt$};

\draw (b4)--(0,-1);
\draw (w4)--(b4);
\draw (0.8,1)--(b4);
\draw (-1.4,2)--(w4);
\draw (-0.2,2)--(w4);

\node at (-1.4,2.2) {$2$};
\node at (-0.2,2.2) {$1$};
\node at (0.8,1.2) {$3$};

\end{scope}

\begin{scope}[shift={(0,-4.2)}] 
\node at (2,0.5) {$+$};

\begin{scope}[shift={(3.5,0)}]

\node[circle,draw,inner sep=2pt] (b1) at (0,0) {$\bt$};
\node[circle,draw,inner sep=2pt] (w1) at (0.8,1) {$\wt$};

\draw (b1)--(0,-1);
\draw (-0.8,1)--(b1);
\draw (w1)--(b1);
\draw (0.2,2)--(w1);
\draw (1.4,2)--(w1);

\node at (-0.8,1.2) {$1$};
\node at (0.2,2.2) {$2$};
\node at (1.4,2.2) {$3$};

\end{scope}

\node at (5.5,0.5) {$-$};

\begin{scope}[shift={(7,0)}]

\node[circle,draw,inner sep=2pt] (b2) at (0,0) {$\bt$};
\node[circle,draw,inner sep=2pt] (w2) at (0.8,1) {$\wt$};

\draw (b2)--(0,-1);
\draw (-0.8,1)--(b2);
\draw (w2)--(b2);
\draw (0.2,2)--(w2);
\draw (1.4,2)--(w2);

\node at (-0.8,1.2) {$2$};
\node at (0.2,2.2) {$1$};
\node at (1.4,2.2) {$3$};

\end{scope}

\node at (9,0.5) {$-$};

\begin{scope}[shift={(11,0)}]

\node[circle,draw,inner sep=2pt] (w3) at (0,0) {$\wt$};
\node[circle,draw,inner sep=2pt] (b3) at (-0.8,1) {$\bt$};

\draw (w3)--(0,-1);
\draw (b3)--(w3);
\draw (0.8,1)--(w3);
\draw (-1.4,2)--(b3);
\draw (-0.2,2)--(b3);

\node at (-1.4,2.2) {$1$};
\node at (-0.2,2.2) {$2$};
\node at (0.8,1.2) {$3$};

\end{scope}

\node at (12.5,0.5) {$+$};

\begin{scope}[shift={(14.5,0)}]

\node[circle,draw,inner sep=2pt] (w4) at (0,0) {$\wt$};
\node[circle,draw,inner sep=2pt] (b4) at (-0.8,1) {$\bt$};

\draw (w4)--(0,-1);
\draw (b4)--(w4);
\draw (0.8,1)--(w4);
\draw (-1.4,2)--(b4);
\draw (-0.2,2)--(b4);

\node at (-1.4,2.2) {$2$};
\node at (-0.2,2.2) {$1$};
\node at (0.8,1.2) {$3$};

\end{scope}

\end{scope}

\end{tikzpicture}
\]

\begin{theorem}\label{thm:IPL}
The operad $\mathcal{IPL}$ is Koszul.
\end{theorem}
\begin{proof}
We shall use Theorem \ref{theorem-Koszulity} and follow the same line of thought used in \cite[6.1]{DG} for the proof of the Koszulity of $\mathcal{PL}$. Consider the operad $\mathcal{L}$ of Lie algebras, and the operad $\mathcal{M}_2 := \mathcal{F}(\wt,\bt)$. It is clear that the latter is Koszul, while it is well-known that $\mathcal{L}$ is Koszul; see e.g.\ \cite[13.2.6]{LodayVallette}. Set $\mathcal{V}=\Bbbk\langle [\cdot,\cdot]\rangle$, $\mathcal{W}=\Bbbk\langle \wt,\wt^{\mathrm{op}},\bt,\bt^{\mathrm{op}} \rangle$, and let $\mathcal{R}=\Bbbk \mathrm{J}$, where $\mathrm{J}$ denotes the Jacobi identity. Consider the $\mathbb{S}$-modules maps $s,d$ of the form \eqref{eq:filtered-law-one}-\eqref{eq:filtered-law-two} given by $s=0$ and 
\begin{align*}
d(\wt \circ_2 [\cdot,\cdot] ) &=  [\cdot,\cdot]\circ_1 \wt + ([\cdot,\cdot]\circ_2\wt ) \cdot(12)\\
d(\wt \circ_1 [\cdot,\cdot]  ) &= (\wt \circ_2 \wt) \cdot (\mathrm{id}_{\mathbb{S}_3} - (12)) + (\wt \circ_1 \wt) \cdot ((12) -\mathrm{id}_{\mathbb{S}_3} ) \\
d(\bt \circ_2 [\cdot,\cdot] ) &=  [\cdot,\cdot]\circ_1 \bt + ([\cdot,\cdot]\circ_2\bt ) \cdot(12)\\
d(\bt\circ_1[\cdot,\cdot]  ) &= (\wt \circ_2 \bt + \bt \circ_2 \wt) \cdot (\mathrm{id}_{\mathbb{S}_3} - (12)) + (\wt \circ_1 \bt + \bt \circ_1 \wt) \cdot ((12) -\mathrm{id}_{\mathbb{S}_3} )
\end{align*}
which we can write in a more algebraic fashion as
\begin{align*}
d(x \wt [y,z]) &= [x \wt y,z] + [y,x \wt z] \\
d([x,y] \wt z )&= x \wt (y \wt z) - y \wt (x \wt z) - (x \wt y) \wt z + (y \wt x) \wt z \\
d(x \bt [y,z]) &= [x \bt y,z] + [y,x \bt z] \\
d([x,y] \bt z) &= \R(x,y,z) - \R(y,x,z) - \L(x,y,z) + \L(y,x,z).
\end{align*}
Employing the same notations of Definition \ref{definition-filtered-law}, we thus have $\mathcal D$ is the $\mathbb S_{3}$-submodule generated by $\mathrm{PL1}$, $\mathrm{PL2}$, $\mathrm{IPL1}$, $\mathrm{IPL2}$, while $\mathcal{G}=\mathcal{R}=\Bbbk\mathrm{J}$ and $\mathcal{S}=0$, where $\mathrm{PL}_1, \mathrm{PL}_2, \mathrm{IPL}_1, \mathrm{IPL}_2$ denote respectively the relations \eqref{eq:postlie:1}, \eqref{eq:postlie:2}, \eqref{eq:infpostlie:1}, \eqref{eq:infpostlie:2}. Therefore, we have that $\mathcal{E} \cong \mathcal{IPL}$. If the corresponding map $\xi_3$ is an isomorphism, then the pair $(s,d)$ defines a filtered distributive law; consequently, we can apply Theorem \ref{theorem-Koszulity} to conclude that $\mathcal{IPL}$ is Koszul, and, in particular, $\mathcal{IPL} \cong \mathcal{L} \circ \mathcal{M}_2$ as $\mathbb{S}$-modules.
To verify that the surjective map $\xi_{3}$ is injective, it remains to check that the following critical pairs\footnote{The term \emph{critical pair} is borrowed from the theory of Gr\"obner basis for quadratic operads. Since the computations that follow are identical to those required to show Koszulity using the Hoffbeck-Dotsenko-Khoroshkin approach, we adopt the same terminology.} are confluent:
\[
[x,y]\bt[z,w],\qquad [[x,y],z]\bt w,\qquad x\bt[[y,z],w]
\]
where we again used an algebraic representation of the corresponding operadic elements.
We observe that the third one follows as in the case of post-Lie algebras; thus, we only have to check the first two. 
We first check the critical pair $[x,y]\bt[z,w]$. The first path (using first the rule \eqref{eq:infpostlie:2}) gives
\begingroup
\allowdisplaybreaks
\begin{align*}
[x,y] \bt [z,w] &= \R(x,y,[z,w]) - \R(y,x,[z,w]) - \L(x,y,[z,w]) + \L(y,x,[z,w]) \\
&= x \wt (y \bt [z,w]) + x \bt(y \wt [z,w]) - y \wt (x \bt [z,w]) - y \bt(x \wt [z,w]) \\
& - (x \wt y) \bt[z,w] - (x \bt y) \wt[z,w] + (y \wt x) \bt[z,w] + (y \bt x) \wt[z,w] \\
&= x \wt [y \bt z,w] + x \wt [z, y \bt w] + x \bt [y \wt z,w] + x \bt [z, y \wt w] \\
& - y \wt [x \bt z,w] - y \wt [z, x \bt w] - y \bt [x \wt z,w] - y \bt [z, x \wt w]\\
&- [(x \wt y)\bt z,w] -[z,(x \wt y)\bt w] - [(x \bt y)\wt z,w] -[z,(x \bt y)\wt w] \\
&+  [(y \wt x)\bt z,w] +[z,(y \wt x)\bt w] + [(y \bt x)\wt z,w] +[z,(y \bt x)\wt w]\\
&=[x \wt (y \bt z),w] + [y \bt z, x \wt w] + [x \wt z, y \bt w] + [z, x \wt (y \bt w)]\\
& + [x \bt (y \wt z),w] + [y \wt z, x \bt w] + [x \bt z, y \wt w] + [z, x \bt (y \wt w)]\\
& - [y \wt (x \bt z),w] - [x \bt z,y \wt w ] - [y \wt z , x \bt w] - [z, y \wt (x \bt w)]\\
&- [y \bt (x \wt z),w] - [x \wt z,y \bt w ] - [y \bt z , x \wt w] - [z, y \bt (x \wt w)]\\
&- [(x \wt y)\bt z,w] -[z,(x \wt y)\bt w] - [(x \bt y)\wt z,w] -[z,(x \bt y)\wt w] \\
&+  [(y \wt x)\bt z,w] +[z,(y \wt x)\bt w] + [(y \bt x)\wt z,w] +[z,(y \bt x)\wt w] \\
&= x \wt [y \bt z,w] + x \wt [z, y \bt w] + x \bt [y \wt z,w] + x \bt [z, y \wt w] \\
& - y \wt [x \bt z,w] - y \wt [z, x \bt w] - y \bt [x \wt z,w] - y \bt [z, x \wt w]\\
&- [(x \wt y)\bt z,w] -[z,(x \wt y)\bt w] - [(x \bt y)\wt z,w] -[z,(x \bt y)\wt w] \\
&+  [(y \wt x)\bt z,w] +[z,(y \wt x)\bt w] + [(y \bt x)\wt z,w] +[z,(y \bt x)\wt w]\\
&=[x \wt (y \bt z),w] + [z, x \wt (y \bt w)] + [x \bt (y \wt z),w]  + [z, x \bt (y \wt w)]\\
& - [y \wt (x \bt z),w] - [z, y \wt (x \bt w)]- [y \bt (x \wt z),w]  - [z, y \bt (x \wt w)]\\
&- [(x \wt y)\bt z,w] -[z,(x \wt y)\bt w] - [(x \bt y)\wt z,w] -[z,(x \bt y)\wt w] \\
&+  [(y \wt x)\bt z,w] +[z,(y \wt x)\bt w] + [(y \bt x)\wt z,w] +[z,(y \bt x)\wt w] .
\end{align*}
\endgroup
The second path (using first the rule \eqref{eq:infpostlie:1}) gives
\begingroup
\allowdisplaybreaks
\begin{align*}
[x,y] \bt [z,w] &= [[x,y] \bt z, w] + [z, [x,y] \bt w] \\
&= [\R(x,y,z),w] - [\R(y,x,z),w] - [\L(x,y,z),w] + [\L(y,x,z),w] \\
& + [z,\R(x,y,w)] - [z,\R(y,x,w)] - [z, \L(x,y,w)] + [z,\L(y,x,w)]\\
&= [x \wt (y \bt z),w] + [x \bt (y \wt z),w] - [y \wt (x \bt z),w] - [y \bt (x \wt z),w]\\
&- [(x \wt y) \bt z,w] - [(x \bt y)\wt z,w] + [(y \wt x) \bt z,w] + [(y \bt x)\wt z,w]\\
& + [z, x \wt (y \bt w)] +  [z, x \bt (y \wt w)] - [z, y \wt (x \bt w)] - [z, y \bt (x \wt w)]\\
& - [z, (x \wt y) \bt w] - [z, (x \bt y) \wt w] + [z, (y \wt x) \bt w] + [z, (y \bt x) \wt w].
\end{align*}
A careful but easy check shows that both paths give the same expansion.
\endgroup

\medskip\medskip
Next, we check the critical pair $[[x,y],z]\bt w$, which turns out to be the most intricate case. We will use \eqref{eq:postlie:1},\eqref{eq:postlie:2},\eqref{eq:infpostlie:1} and \eqref{eq:infpostlie:2}, rewritten by using the operators $D_{x}:=x\wt(-)$ and $E_{x}:=x\bt(-)$, and by employing notations $\ab{x}{y} := x\wt y-y\wt x$ and $\abd{x}{y} := x\bt y-y\bt x$. These have the following form:
\begin{align}
D_x[y,z] &= [D_xy,z]+[y,D_xz],\label{PL1'}\\
D_{[x,y]} &= [D_x,D_y]-D_{\ab{x}{y}},\label{PL2'}\\
E_x[y,z] &= [E_xy,z]+[y,E_xz],\label{IPL1'}\\
E_{[x,y]} &= [D_x,E_y]+[E_x,D_y]-E_{\ab{x}{y}}-D_{\abd{x}{y}}\label{IPL2'}.
\end{align} 

We have $[[x,y],z]\bt w = E_{[[x,y],z]}(w)$, hence we study $E_{[[x,y],z]}$. \medskip

The first path $(A)$ uses \eqref{IPL2'}. Hence, we have
\begin{equation}\label{eq1}
E_{[[x,y],z]}\overset{\eqref{IPL2'}}{=}\underbrace{[D_{[x,y]},E_z]+[E_{[x,y]},D_z]}+(\underbrace{
   -E_{\ab{[x,y]}{z}}-D_{\abd{[x,y]}{z}}})=:\mathcal{P}_A+\mathcal{S}_A .
\end{equation}
We compute
\begin{align*}
[D_{[x,y]},E_z] &\overset{\eqref{PL2'}}{=} \bigl[[D_x,D_y],E_z\bigr]-[D_{\ab{x}{y}},E_z],\\
[E_{[x,y]},D_z] &\overset{\eqref{IPL2'}}{=} \bigl[[D_x,E_y],D_z\bigr]+\bigl[[E_x,D_y],D_z\bigr]-[E_{\ab{x}{y}},D_z]-[D_{\abd{x}{y}},D_z].
\end{align*}
We also have
\begin{align*}
\ab{[x,y]}{z}  &= [x,y]\wt z-z\wt[x,y]\overset{\eqref{PL1'}}{=} D_{[x,y]}z-\bigl([z\wt x,y]+[x,z\wt y]\bigr),\\
\abd{[x,y]}{z} &= [x,y]\bt z-z\bt[x,y]\overset{\eqref{IPL1'}}{=} E_{[x,y]}z-\bigl([z\bt x,y]+[x,z\bt y]\bigr),
\end{align*}
so that 
\begin{align*}
  -E_{\ab{[x,y]}{z}}  &= -E_{D_{[x,y]}z}+E_{[z\wt x,\,y]}+E_{[x,\,z\wt y]},\\
  -D_{\abd{[x,y]}{z}} &= -D_{E_{[x,y]}z}+D_{[z\bt x,\,y]}+D_{[x,\,z\bt y]}.
\end{align*}
Furthermore, we have
\begin{align*}
E_{[z\wt x,y]} &\overset{\eqref{IPL2'}}{=} [D_{z\wt x},E_y]+[E_{z\wt x},D_y]-E_{\ab{z\wt x}{y}}-D_{\abd{z\wt x}{y}},\\
E_{[x,z\wt y]} &\overset{\eqref{IPL2'}}{=}[D_x,E_{z\wt y}]+[E_x,D_{z\wt y}]-E_{\ab{x}{z\wt y}}-D_{\abd{x}{z\wt y}},\\
D_{[z\bt x,y]} &\overset{\eqref{PL2'}}{=} [D_{z\bt x},D_y]-D_{\ab{z\bt x}{y}},\\
D_{[x,z\bt y]} &\overset{\eqref{PL2'}}{=} [D_x,D_{z\bt y}]-D_{\ab{x}{z\bt y}}.
\end{align*}
We also compute
\begin{align*}
D_{[x,y]}z &\overset{\eqref{PL2'}}{=} x\wt(y\wt z)-y\wt(x\wt z)-\ab{x}{y}\wt z,\\
E_{[x,y]}z &\overset{\eqref{IPL2'}}{=} x\wt(y\bt z)+x\bt(y\wt z)-y\wt(x\bt z)-y\bt(x\wt z)-\abd{x}{y}\wt z-\ab{x}{y}\bt z,
\end{align*}
hence \eqref{eq1} reads $E_{[[x,y],z]}=\mathcal{P}_A+\mathcal{S}_A$, where:
\begin{align*}
\mathcal P_A&= \underbrace{\bigl[[D_x,D_y],E_z\bigr]+\bigl[[D_x,E_y],D_z\bigr]+\bigl[[E_x,D_y],D_z\bigr]}_{\text{(i)}}\quad\underbrace{-\,[D_{\ab{x}{y}},E_z]-[E_{\ab{x}{y}},D_z]-[D_{\abd{x}{y}},D_z]}_{\text{(ii)}}\\[2pt]
&\quad \underbrace{+\,[D_{z\wt x},E_y]+[E_{z\wt x},D_y]+[D_{z\bt x},D_y]+[D_x,E_{z\wt y}]+[E_x,D_{z\wt y}]+[D_x,D_{z\bt y}]}_{\text{(iii)}}\, 
\end{align*}
and 
\begin{align*}
\mathcal S_A &= -E_{\,D_{[x,y]}z\,+\,\ab{z\wt x}{y}\,+\,\ab{x}{z\wt y}}-D_{\,E_{[x,y]}z\,+\,\abd{z\wt x}{y}+\abd{x}{z\wt y}+\ab{z\bt x}{y}+\ab{x}{z\bt y}} .
\end{align*}

\medskip
The second path $(B)$ uses Jacobi. Hence, we have $ E_{[[x,y],z]}=E_{[x,[y,z]]}-E_{[y,[x,z]]}$. We compute the first addend; the second will be obtained by interchanging $x$ with $y$. We have
\begin{equation}\label{eq2}
E_{[x,[y,z]]}\overset{\eqref{IPL2'}}{=}[D_x,E_{[y,z]}]+[E_x,D_{[y,z]}]+(-E_{\ab{x}{[y,z]}}-D_{\abd{x}{[y,z]}}).
\end{equation}
We will denote by $\mathcal{P}_{B}$ the first two addends of \eqref{eq2} together with the other two coming from interchanging $x$ with $y$, and by $\mathcal{S}_{B}$ the third and fourth addends and the two coming from interchanging $x$ with $y$.

\begin{align*}
[D_x,E_{[y,z]}] &\overset{\eqref{IPL2'}}{=}\bigl[D_x,[D_y,E_z]\bigr]+\bigl[D_x,[E_y,D_z]\bigr]-[D_x,E_{\ab{y}{z}}]-[D_x,D_{\abd{y}{z}}],\\
  [E_x,D_{[y,z]}] &\overset{\eqref{PL2'}}{=} \bigl[E_x,[D_y,D_z]\bigr]-[E_x,D_{\ab{y}{z}}].
\end{align*}
We also compute 
\begin{align*}
  \ab{x}{[y,z]}  &\overset{\eqref{PL1'}}{=} x\wt[y,z]-[y,z]\wt x=[x\wt y,z]+[y,x\wt z]-D_{[y,z]}x,\\
  \abd{x}{[y,z]} &\overset{\eqref{IPL1'}}{=} x\bt[y,z]-[y,z]\bt x=[x\bt y,z]+[y,x\bt z]-E_{[y,z]}x,
\end{align*}
so that
\begin{align*}
  -E_{\ab{x}{[y,z]}}  &= -E_{[x\wt y,z]}-E_{[y,x\wt z]}+E_{D_{[y,z]}x},\\
  -D_{\abd{x}{[y,z]}} &= -D_{[x\bt y,z]}-D_{[y,x\bt z]}+D_{E_{[y,z]}x}.
\end{align*}
Moreover, we have
\begin{align*}
E_{[x\wt y,z]} &\overset{\eqref{IPL2'}}{=} [D_{x\wt y},E_z]+[E_{x\wt y},D_z]-E_{\ab{x\wt y}{z}}-D_{\abd{x\wt y}{z}},\\
E_{[y,x\wt z]} &\overset{\eqref{IPL2'}}{=}[D_y,E_{x\wt z}]+[E_y,D_{x\wt z}]-E_{\ab{y}{x\wt z}}-D_{\abd{y}{x\wt z}},\\
D_{[x\bt y,z]} &\overset{\eqref{PL2'}}{=} [D_{x\bt y},D_z]-D_{\ab{x\bt y}{z}},\\
D_{[y,x\bt z]} &\overset{\eqref{PL2'}}{=} [D_y,D_{x\bt z}]-D_{\ab{y}{x\bt z}} .
\end{align*}
After subtracting the expansion with $x$ and $y$ interchanged, we write $E_{[[x,y],z]}=\mathcal P_{B}+\mathcal S_{B}$, grouping the commutator terms and the remaining terms as in path $(A)$. Thus,
\begin{align*}
\mathcal P_B&= \Bigl\{\bigl[D_x,[D_y,E_z]\bigr]+\bigl[D_x,[E_y,D_z]\bigr]+\bigl[E_x,[D_y,D_z]\bigr]\\
&\qquad -[D_x,E_{\ab{y}{z}}]-[D_x,D_{\abd{y}{z}}]-[E_x,D_{\ab{y}{z}}]\\
&\qquad -[D_{x\wt y},E_z]-[E_{x\wt y},D_z]-[D_{x\bt y},D_z]\\
  &\qquad -[D_y,E_{x\wt z}]-[E_y,D_{x\wt z}]-[D_y,D_{x\bt z}]
     \Bigr\}-\bigl(x\leftrightarrow y\bigr),
\end{align*}
and
\[
  \mathcal S_B
  = E_{\,\Pi(x,y)-\Pi(y,x)}\;+\;D_{\,\Pi^{\bullet}(x,y)-\Pi^{\bullet}(y,x)},
\]
where
\begin{align*}
\Pi(x,y)&:= D_{[y,z]}x+\ab{x\wt y}{z}+\ab{y}{x\wt z},\\
\Pi^{\bullet}(x,y)&:= E_{[y,z]}x+\abd{x\wt y}{z}+\abd{y}{x\wt z}+\ab{x\bt y}{z}+\ab{y}{x\bt z}.
\end{align*}
We now compare path $(A)$ with path $(B)$. By Jacobi, we have that the addends in $(i)$ of $\mathcal{P}_{A}$ are :
\begin{align*}
\bigl[[D_x,D_y],E_z\bigr] &= \bigl[D_x,[D_y,E_z]\bigr]-\bigl[D_y,[D_x,E_z]\bigr],\\
\bigl[[D_x,E_y],D_z\bigr] &= \bigl[D_x,[E_y,D_z]\bigr]-\bigl[E_y,[D_x,D_z]\bigr],\\
\bigl[[E_x,D_y],D_z\bigr] &= \bigl[E_x,[D_y,D_z]\bigr]-\bigl[D_y,[E_x,D_z]\bigr].
\end{align*}
The six terms on the right are exactly the first three in $\mathcal{P}_{B}$ and the three obtained by interchanging $x$ with $y$. From the third line of $\mathcal{P}_{B}$ with the corresponding terms obtained by interchanging $x$ with $y$, we get
\[
-[D_{x\wt y-y\wt x},E_z]-[E_{x\wt y-y\wt x},D_z]-[D_{x\bt y-y\bt x},D_z]=-[D_{\ab{x}{y}},E_z]-[E_{\ab{x}{y}},D_z]-[D_{\abd{x}{y}},D_z],
\]
which is $(ii)$ in $\mathcal{P}_{A}$. From the second line of $\mathcal{P}_{B}$ with the corresponding terms obtained interchanging $x$ with $y$ we get
\begin{align*}
&-[D_x,E_{y\wt z}]+[D_x,E_{z\wt y}]-[D_x,D_{y\bt z}]+[D_x,D_{z\bt y}]-[E_x,D_{y\wt z}]+[E_x,D_{z\wt y}]\\&+[D_y,E_{x\wt z}]-[D_y,E_{z\wt x}]+[D_y,D_{x\bt z}]-[D_y,D_{z\bt x}]+[E_y,D_{x\wt z}]-[E_y,D_{z\wt x}],
\end{align*}
while the fourth line of $\mathcal{P}_{B}$ with the corresponding terms obtained interchanging $x$ with $y$ gives
\[
-[D_y,E_{x\wt z}]-[E_y,D_{x\wt z}]-[D_y,D_{x\bt z}]\;+\;[D_x,E_{y\wt z}]+[E_x,D_{y\wt z}]+[D_x,D_{y\bt z}] .
\]
Summing with the previous one, it remains
\[
[D_x,E_{z\wt y}]+[D_x,D_{z\bt y}]+[E_x,D_{z\wt y}]-[D_y,E_{z\wt x}]-[D_y,D_{z\bt x}]-[E_y,D_{z\wt x}]
\]
which is $(iii)$ in $\mathcal{P}_{A}$. We have shown that $\mathcal{P}_{A}=\mathcal{P}_{B}$.

From the relations for post-Lie algebras \cite[Theorem 6.1]{DG}, we already know that 
\[
-D_{[x,y]}z-\ab{z\wt x}{y}-\ab{x}{z\wt y}=\Pi(x,y)-\Pi(y,x).
\]
Finally, we have

\begin{align*}
-E_{[x,y]}z-&\abd{z\wt x}{y}-\abd{x}{z\wt y}-\ab{z\bt x}{y}-\ab{x}{z\bt y}\\
&\overset{\eqref{IPL2'}}{=}-x\wt(y\bt z)-x\bt(y\wt z)+y\wt(x\bt z)+y\bt(x\wt z)\\
&+(x\bt y)\wt z-(y\bt x)\wt z+(x\wt y)\bt z-(y\wt x)\bt z\\
&-(z\wt x)\bt y+y\bt(z\wt x)-x\bt(z\wt y)+(z\wt y)\bt x\\
&-(z\bt x)\wt y+y\wt(z\bt x)-x\wt(z\bt y)+(z\bt y)\wt x\\&=\Pi^{\bullet}(x,y)-\Pi^{\bullet}(y,x).
\end{align*}
Therefore, $\mathcal{S}_{A}=\mathcal{S}_{B}$. Hence, all critical pairs are confluent, and $\xi_3$ is an isomorphism, concluding the proof. 
\end{proof}

\section{Outlook}\label{sec:outlook}
We conclude this article with a selection of prospective applications of our results that we intend to investigate in the near future. \\ \\
\textbf{Quantization of post-Lie bialgebras}: recall that a Lie bialgebra is a triple $(\mathfrak{b},[\cdot,\cdot],\delta)$, where $(\mathfrak{b},[\cdot,\cdot])$ is a Lie algebra, $(\mathfrak{b},\delta)$ is a Lie coalgebra, and the cocycle condition $\delta([x,y])= x \cdot\delta(y) - y \cdot \delta(x)$ holds for all $x,y \in \mathfrak{b}$. A topological Hopf algebra $(H,\Delta)$ is said to be a quantization of a Lie bialgebra $(\mathfrak{b},[\cdot,\cdot],\delta)$ if there is a Hopf algebra isomorphism $H/\hbar \cdot H \cong \mathrm{U}(\mathfrak{b})$ such that 
\[ \delta(x) = \frac{\Delta(\tilde{x}) - \Delta^{\mathrm{op}}(\tilde{x})}{\hbar} \ \mod \hbar\]
where $\tilde{x}$ is any lift of $x$ in $H$. The problem of quantization of Lie bialgebras was posed by V. Drinfeld \cite{Drinfeld90} and solved by P. Etingof and D. Kazhdan \cite{EK} (see also P. \v{S}evera's solution \cite{Sev}). Since the universal enveloping algebra of a post-Lie algebra carries a natural structure of a post-Hopf algebra (see Theorem \ref{thm:postHopfonuniversalenveloping}), it is natural to ask whether post structures are compatible with quantization. We expect that, if  $(\mathfrak{b},[\cdot,\cdot],\delta)$ is a Lie bialgebra together with a \emph{compatible} post-Lie structure  $(\mathfrak{b},[\cdot,\cdot],\wt)$, there is a topological post-Hopf algebra $(H,\Delta,\wp)$ quantizing  $(\mathfrak{b},[\cdot,\cdot],\delta)$ and such that the \emph{semiclassical limit} of $\wp$ is $\wt$. \\ \\
\textbf{Infinitesimal post-Lie-Rinehart algebras and post-Hopf algebroids}: two well-established generalisations of the concepts of Lie algebra and Hopf algebra are the well-known notions of Lie-Rinehart algebra \cite{Rinehart} and of Hopf algebroid \cite{Lu}. 
Recently, A. B. Laurent, Y. Li, and Y. Sheng \cite{LLS} proposed an \emph{oid} version of post-Lie algebras and post-Hopf algebras. They also showed that the universal enveloping algebra of a post-Lie-Rinehart algebra carries a natural structure of a post-Hopf algebroid, and, moreover, they provided the construction of the free post-Lie-Rinehart algebra. \\
A very natural follow-up would be to use the same techniques to define the notions of infinitesimal post-Lie-Rinehart algebra and infinitesimal post-Hopf algebroid, and to discuss their geometric interpretation.\\ \\

\textbf{Classification of infinitesimal post-Lie algebra structures on $\mathfrak{t}(2,\mathbb{C})$ and $\mathfrak{gl}(2,\mathbb{C})$}: In Section \ref{subsection-inf-post-lie-sl2}, we have classified infinitesimal post-Lie structures on the Lie algebra $\mathfrak{sl}(2)$. Our classification builds upon the earlier classification of post-Lie structures obtained by Y. Pan, Q. Liu, C. Bai, and L. Guo \cite[Theorem 1.2]{PLBG}. Similar classification results are available for other Lie algebras. For instance, post-Lie structures on the Lie algebra $\mathfrak{t}(2,\mathbb{C})$ of complex upper triangular matrices were classified by X. Tang and Y. Zhang \cite{TZ}, while those on $\mathfrak{gl}(2,\mathbb{C})$ were classified by Y. Sheng and X. Tang \cite{ST}. A natural direction for future research is to use these classifications to determine the corresponding infinitesimal post-Lie structures on $\mathfrak{t}(2,\mathbb{C})$ and $\mathfrak{gl}(2,\mathbb{C})$.\\ \\

\noindent\textit{Acknowledgments.} 
A.\ Rivezzi and A.\ Sciandra are members of the ``National Group for Algebraic and Geometric Structures and their Applications'' (GNSAGA-INdAM). A.\ Rivezzi is supported by GA\v{C}R/NCN grant Quantum Geometric Representation Theory and Noncommutative Fibrations 24-11728K.
A. Sciandra is supported by a postdoctoral fellowship at the ULB within the framework of the PDR project ``Reconstruction of modules and algebraic objects from closed and monoidal structures on their representation categories'' funded by the FNRS under the grant number T.0318.25F (PI J. Vercruysse).
T. Weber is supported by the GA\v{C}R
PIF 24-11324I. This publication is based upon work from COST Action CaLISTA CA21109 supported by COST
(European Cooperation in Science and Technology), www.cost.eu.

\noindent\textbf{Use of AI.} ChatGPT-6 (Open AI) was used for a final proof-reading of the work and suggested a correction for Theorem \ref{thm:sl2-infinitesimal-classification} and other minor adjustments. It was also used to improve Section \ref{sec:geometry}. The authors verified all AI-assisted material and take full responsibility for the final text.

\bigskip
\noindent\textbf{Contacts}
\medskip

\noindent
Andrea Rivezzi:
\texttt{andrea.rivezzi@matfyz.cuni.cz}\\
\url{https://sites.google.com/view/andrearivezzipersonalwebpage/}

\medskip
\noindent
Andrea Sciandra: \texttt{andrea.sciandra@ulb.be}\\
\url{www.andreasciandra.com}

\medskip
\noindent
Thomas Weber: \texttt{thomas.weber@matfyz.cuni.cz}\\
\url{https://sites.google.com/view/thomasweber}

\end{document}